\pdfoutput=1
\documentclass[hidelinks, reqno]{amsart}
\usepackage[english,french]{babel}
\usepackage{graphicx}
\usepackage{subcaption}
\usepackage{tabularx}
\usepackage{booktabs}
\usepackage{array}
\usepackage{amsmath}
\usepackage{amsfonts}
\usepackage{amssymb}
\usepackage{amsthm}
\usepackage{dsfont}
\usepackage{appendix}
\usepackage[scr]{rsfso}
\usepackage{enumitem}
\usepackage{permute}
\usepackage{slashed}
\usepackage[usenames,dvipsnames]{xcolor}
\usepackage[pagebackref = true, colorlinks, linkcolor = Red, citecolor = Green, bookmarksdepth=2, linktocpage=true]{hyperref}
\usepackage[capitalize]{cleveref}
\usepackage{comment}
\usepackage{changepage}

\usepackage[T1]{fontenc}
\usepackage{makecell}

\usepackage{tikz}
\usetikzlibrary{calc}

\allowdisplaybreaks

\DeclareMathOperator{\Id}{Id}

\DeclareMathOperator{\vol}{vol}
\DeclareMathOperator{\supp}{supp}

\DeclareMathOperator{\SL}{SL}

\DeclareMathOperator{\ad}{ad}
\DeclareMathOperator{\Ad}{Ad}

\DeclareMathOperator{\inj}{inj}

\DeclareMathOperator{\im}{im}

\newcommand{\N}{\mathbb{N}}
\newcommand{\Z}{\mathbb{Z}}
\newcommand{\Q}{\mathbb{Q}}
\newcommand{\R}{\mathbb{R}}

\newcommand{\LieG}{\mathfrak{g}}

\newcommand{\LieN}{\mathfrak{n}}

\newcommand{\LieH}{\mathfrak{h}}

\newcommand{\LieL}{\mathfrak{l}}

\newcommand{\height}{\mathrm{ht}}
\newcommand{\disc}{\mathrm{disc}}

\makeatletter
\newcommand{\mytag}[2]{%
\text{#1}%
\@bsphack
\begingroup
\@onelevel@sanitize\@currentlabelname
\edef\@currentlabelname{%
\expandafter\strip@period\@currentlabelname\relax.\relax\@@@%
}%
\protected@write\@auxout{}{%
\string\newlabel{#2}{%
{#1}%
{\thepage}%
{\@currentlabelname}%
{\@currentHref}{}%
}%
}%
\endgroup
\@esphack
}
\makeatother

\DeclareFontFamily{U}{mathb}{\hyphenchar\font45}
\DeclareFontShape{U}{mathb}{m}{n}{
<5> <6> <7> <8> <9> <10> gen * mathb
<10.95> mathb10 <12> <14.4> <17.28> <20.74> <24.88> mathb12
}{}
\DeclareSymbolFont{mathb}{U}{mathb}{m}{n}
\DeclareMathSymbol{\bigast}{2}{mathb}{"06}

\def\XXint#1#2#3{{\setbox0=\hbox{$#1{#2#3}{\int}$}
\vcenter{\hbox{$#2#3$}}\kern-.5\wd0}}

\theoremstyle{plain}
\newtheorem{theorem}{Theorem}[section]
\newtheorem{proposition}[theorem]{Proposition}
\newtheorem{lemma}[theorem]{Lemma}
\newtheorem{corollary}[theorem]{Corollary}

\theoremstyle{definition}

\theoremstyle{remark}
\newtheorem{remark}[theorem]{Remark}

\Crefname{enumi}{Property}{Properties}
\Crefname{alternativei}{Alternative}{Alternatives}
\Crefname{subsection}{Subsection}{Subsections}

\begin{document}
\selectlanguage{english}


\title[Polynomial effective density in quotient of $\SL_2(\Q_p) \times \SL_2(\Q_p)$]{Polynomial effective density in quotient of $\SL_2(\Q_p) \times \SL_2(\Q_p)$}

\author{Zuo Lin}
\address{Department of Mathematics, UC San Diego, 9500 Gilman Drive, La Jolla, CA 92093, USA}
\email{zul003@ucsd.edu}

\date{\today}

\begin{abstract}
We prove an effective density theorem with polynomial error rate for orbits of upper triangular subgroup of $\SL_2(\Q_p)$ in $\SL_2(\Q_p) \times \SL_2(\Q_p)$ for prime number $p > 3$. The proof is based on the use of Margulis function, a restricted projection theorem on $\Q_p^3$, and spectral gap of the ambient space.  
\end{abstract}

\maketitle

\setcounter{tocdepth}{1}

\section{Introduction}

In this paper, we prove an effective density theorem with polynomial error rate for orbits of upper triangular subgroup of $\SL_2(\Q_p)$ in $\SL_2(\Q_p) \times \SL_2(\Q_p)$ for prime number $p > 3$. This is an analogue to the main result in \cite{LM23} in $p$-adic case. We will prove an effective equidistribution result in a forthcoming paper. For history and recent development of effective density or equidistribution results with polynomial error rate, the reader could consult \cite{LM23}, \cite{LMW22}, \cite{LMW23}, \cite{LMWY23} and \cite{Yan23}. 

Now we fix some notations to state the main result. In this paper, we will always use $p$ to denote a prime number with $p > 3$. Let
\begin{align*}
    G = \SL_2(\Q_p) \times \SL_2(\Q_p)
\end{align*}
and
\begin{align*}
    H = \{(g, g): g \in \SL_2(\Q_p)\} \cong \SL_2(\Q_p).
\end{align*}
Let $\Gamma$ be a lattice in $G$ and put $X = G /\Gamma$. Let $P$ be the group of upper triangular matrices in $H$.  

Let
\begin{align*}
    K = \SL_2(\Z_p) \times \SL_2(\Z_p)
\end{align*}
and 
\begin{align*}
    K[n] = \ker(\SL_2(\Z_p) \times \SL_2(\Z_p) \to \SL_2(\Z/p^n\Z) \times \SL_2(\Z/p^n\Z)).
\end{align*}
$\{K[n]\}_{n \in \N}$ form a basis of neighborhood of $e \in G$. For a positive real number $r > 0$, we define $K[r] = K[\lfloor r \rfloor]$. 

An orbit $H.x \subset X$ is periodic if $H \cap \mathrm{Stab}(x)$ is a lattice in $H$. For the semisimple group $H$, $H.x \subset X$ is periodic if and only if $H.x$ is closed in $X$. 

Let $|\cdot|_p$ be the $p$-adic absolute value on $\Q_p$ with uniformizer $p$ and let $|\cdot|$ be the Haar measure on $\Q_p$ with $|\Z_p| = 1$. Let $\|\cdot\|_p$ be the maximum norm on $\mathrm{Mat}_2(\Q_p) \times \mathrm{Mat}_2(\Q_p)$ with respect to the standard basis. For every $T > 0$ and subgroup $L \subset G$, let 
\begin{align*}
    B_{L}(e, T) = \{g \in L: \|g - I\|_p \leq T\}.
\end{align*}
Note that $K[n] = B_{K}(e, p^{-n})$. We will also use $L_T$ to denote $B_L(e, T)$ in this paper. 

The following is the main theorem of the paper. It is a $p$-adic analogue to \cite[Theorem 1.1]{LM23}.
\begin{theorem}\label{thm:Main}
    Suppose $\Gamma$ is an arithmetic lattice. For every $0 < \delta < \frac{1}{2}$, every $x_0 \in X$, every $p^N \gg_{\inj (X)} 1$, at least one of the following holds. 
    \begin{enumerate}
        \item For every $x \in X$, we have
        \begin{align*}
            K[\kappa_0 \delta N - C_0].x \cap B_P(e, p^{A_0 N}).x_0 \neq \emptyset.
        \end{align*}
        \item There exists $x' \in X$ such that $H. x'$ is periodic with $\vol(H.x') \leq p^{\delta N}$, and
        \begin{align*}
            x' \in K[N - C_0].x_0.
        \end{align*}
    \end{enumerate}
    The constants $\kappa_0$, $A_0$, and $C_0$ are positive constant depending only on $(G, H, \Gamma)$. 
\end{theorem}

\cref{thm:Main} follows from the following proposition. 

Let
\begin{align*}
    N = \biggl\{n(r, s) = \biggl(\begin{pmatrix}
        1 & r + s\\
         & 1
    \end{pmatrix}, \begin{pmatrix}
        1 & r\\
         & 1
    \end{pmatrix}\biggr): r, s \in \Q_p\biggr\}
\end{align*}
and $U = \{n(r, 0) : r \in \Q_p\}$. We will use $u_r$ to denote $n(r, 0)$. Let $V = \{n(0, s): s \in \Q_p\}$. We will use $v_s$ to denote $n(0, s)$. We have $N = UV$. 

\begin{proposition}\label{pro:Main}
    Suppose $\Gamma$ is an arithmetic lattice. There exists some $\eta_0 > 0$ depending on $X$ with the following property. 

    Let $0 < \theta, \delta < \frac{1}{2}$, $0 < \eta < \eta_0$, and $x_0 \in X$. There are $\kappa_1$ and $A_1$ depending on $\theta$, and $N_1$ depending on $\delta, \eta$ so that for all $N > N_1$ at least one of the following holds. 
    \begin{enumerate}
        \item There exists a finite subset $I \subset \Z_p$ so that both of the following are satisfied. 
        \begin{enumerate}
            \item The set $I$ supports a probability measure $\rho$ which satisfies
            \begin{align*}
                \rho(J) \leq C_{\theta} |J|^{1 - \theta}
            \end{align*}
            for every open subgroup $J \subset \Z_p$ with $|J| \geq p^{-\delta \kappa_1N}$ where $C_{\theta} \geq 1$ only depends on $\theta$. 
            \item There is a point $y_0 \in X$ so that 
            \begin{align*}
                B_{P}(e, T^{A_1}).x_0 \cap K[\delta \kappa_1 N - C_1]v_s.y_0
            \end{align*}
            for all $s \in I \cup \{0\}$. 
        \end{enumerate}
        \item There exists $x \in X$ so that $H.x$ is periodic with $\vol(H.x) \leq p^{\delta N}$ and 
        \begin{align*}
            x \in K[N- C_1].x_0.
        \end{align*}
    \end{enumerate}
    The constant $C_1$ depends only on $X$. 
\end{proposition}

\cref{thm:Main} follows from \cref{pro:Main} by an argument due to Venkatesh \cite{Ven10}. See \cite{Ven10} or \cite[Section 4]{LM23} for detailed discussion. \cref{sec:Venkatesh} is devoted to this argument in our case. 

The proof of \cref{pro:Main} follows a similar strategy to \cite{LM23}. We provide a sketch of the idea of the proof here. For a comprehensive outline of this strategy, the reader could consult \cite[Section 1]{LM23} or \cite[Section 1]{LMW23}. 

\begin{enumerate}[label=Step \arabic*.]
    \item By working with a small thickening in $H$-direction of orbit $P.x_0$, we show that either case~(2) in \cref{pro:Main} holds, or we can find some point $x$ in that thickening of $B_{P}(e, p^{O(\delta N)}).x_0$ so that any two nearby points in have distance $ > p^{-N}$ transversal. This is done in \cref{sec:ArithmeticLatticesClosedHOrbitVolume} and \cref{sec:ClosingLemma}. \cref{sec:ArithmeticLatticesClosedHOrbitVolume} provides an estimate of volume of a closed $H$-orbit via arithmetic informations following methods in \cite{EMV09} and \cite{ELMV11}. \cref{sec:ClosingLemma} proves an effective closing lemma. 
    \item Assuming case~(2) in \cref{pro:Main} does not hold, we use a Margulis function to show that the translate the thickening of $B_{P}(e, p^{O(\delta N)}).x_0$ in step~1 by a random element of $B_{P}(e, p^{O_{\theta}(N)})$ have dimension $1 - \theta$ transverse to $H$ at scale $p^{-O(\delta N )}$. This step is done in \cref{sec:MargulisFunctions}. The proof is similar to \cite[Section 7]{LM23}. 
    \item In the third step, we use a restricted projection theorem in $\mathbb{Q}_p^3$ discussed below with some arguments in homogeneous dynamics, to project the aforementioned dimension to the direction $V$. 
\end{enumerate}

We indicate the main difference here. One of the ingredients in the final step of \cite{LM23} is a restricted projection theorem from incidence geometry, \cite[Theorem 5.2]{LM23}. It is a finitary version of \cite[Theorem 1.2]{KOV21}. The proof is based on the works of Wolff and Schlag, \cite{Wol00}, \cite{Sch03} using an incidence estimate on circles in Euclidean space following from a cell decomposition theorem due to Clarkson, Edelsbrunner, Guibas, Sharir,
and Welzl, \cite{CEHGLSW90}. 

However, the arguments in \cite{KOV21} (see also \cite[Appendix B]{LM23}) relied on an incidence bound for circles in $\R^3$ which does not hold in $\Q_p^3$. Recently, Gan, Guo and Wang proved a restricted projection theorem in $\R^n$ using decoupling inequality, \cite{GGW24}. 

Our proof of \cref{thm:Main} uses a restricted projection theorem (\cref{thm:ProjDecoupling2}) in $\Q_p^3$ proved in \cite{JL}. We will also state it after we introduce some notations in this section. The proof of it is similar to \cite{GGW24}, which make use of decoupling inequality for moment curve in $\Q_p^n$, see \cite{JL}. We remark here that the restricted projection theorem used in this paper could also be proved using decoupling of cone over parabola with methods in \cite{GGGHMW24}.  

Now we introduce some notations to state the restricted projection theorem. Let $\mathfrak{r} = \mathfrak{sl}_2(\Q_p) \oplus \{0\}$. Throughout this paper, we will always use the following notation for elements $w \in \mathfrak{r}$: 
\begin{align*}
    w = \begin{pmatrix}
        w_{11} & w_{12}\\
        w_{21} & -w_{11}
    \end{pmatrix}
\end{align*}
where $w_{ij} \in \Q_p$. 

Note that in the above coordinate, we have the following expression of $\Ad_{u_r} w$ for all $r \in \Q_p$ and $w \in \mathfrak{r}$: 
\begin{align*}
    \Ad_{u_r} w = \begin{pmatrix}
        w_{11} + w_{21}r & w_{12} - 2w_{11}r - w_{21}r^2\\
        w_{21} & - w_{11} - w_{21}r
    \end{pmatrix}.
\end{align*}

Let $\xi_r(w) = \bigl(\Ad_{u_r} w\bigr)_{12} = w_{12} - 2w_{11}r - w_{21}r^2$ and view it as a $1$-parameterized family of projections form $\Q_p^3$ to $\Q_p$, we have the following restricted projection theorem. 

\begin{theorem}\label{thm:ProjDecouplingIntro}
    Let $p > 3$ be a prime number. Let $0 < \alpha < 1$, $0 < b_0 = p^{-l_0} < b_1 = p^{-l_1} < 1$ be three parameters. Let $E \subset B_{\mathfrak{r}}(0, b_1)$ be so that 
    \begin{align*}
        \frac{\#(E \cap B_{\mathfrak{r}}(w, b))}{\#E} \leq D\cdot (\frac{b}{b_1})^{\alpha}
    \end{align*}
    for all $w \in \mathfrak{r}$ and all $b \geq b_0$, and some $D \geq 1$. Let $0 < \epsilon < 10^{-70}$ and let $J$ be a ball in $\Z_p$. Let $\xi_r$ be the following map: 
    \begin{align*} 
        \xi_r(w) = \bigl(\Ad_{u_r}(w)\bigr)_{12} = w_{12} - 2w_{11} r - w_{21} r^2. 
    \end{align*}
    There exists $J' \subset J$ such that $|J'| \geq (1 - \frac{1}{p})|J|$ satisfying the following. Let $r \in J'$, then there exists a subset $E_r \subset E$ with 
    \begin{align*}
        \#E_r \geq (1 - \frac{1}{p}) \cdot (\#E)
    \end{align*}
    such that for all $w \in E_r$ and all $b \geq b_0$, we have
    \begin{align*}
        \frac{\#\{w' \in E: |\xi_r(w') - \xi_r(w)|_p \leq b\}}{\#E} \leq C_{\epsilon}\cdot(\frac{b}{b_1})^{\alpha - \epsilon}.
    \end{align*}
    where $C_{\epsilon}$ depends on $\epsilon$, $|J|$, $D$. 
\end{theorem}

\subsection*{Acknowledgements}
We thank Amir Mohammadi for suggesting this problem and many useful conversations. 

\section{Preliminaries}\label{sec:Preliminaries}

\subsection{Notations}
Let $G$, $H$, $\Gamma$, $U$, $N$, and $V$ be as in the introduction. Let $X = G/\Gamma$. 

Let $K_H = K \cap H$. Let $K_H[n] = \ker(\SL_2(\Z_p) \to \SL_2(\Z/p^n\Z))$. For a positive real number $r > 0$, let $K_H[r] = K_{H}[\lfloor r \rfloor]$. Note that $K_H[r] = B_{K_H}(e, p^{-r})$. We will also use $K_{H, \beta}$ to denote $B_{K_H}(e, \beta)$. 

Let 
\begin{align*}
    U^{-} = \biggl\{u^{-}_r = \biggl(\begin{pmatrix}
    1 & \\
    r & 1
\end{pmatrix}, \begin{pmatrix}
    1 & \\
    r & 1
\end{pmatrix}\biggr): r \in \Q_p \biggr\},
\end{align*}
and
\begin{align*}
    D = \biggl\{d_{\lambda} = \biggl(\begin{pmatrix}
        \lambda & \\
         & \lambda^{-1}
    \end{pmatrix}, \begin{pmatrix}
        \lambda & \\
         & \lambda^{-1}
\end{pmatrix}\biggr): \lambda \in \Q_p \backslash \{0\}\biggr\}.
\end{align*}
We will use $a_n$ to denote $d_{p^{-n}}$ for simplicity. 

Let $U[n] = \{u_r: r \in p^n\Z_p\}$, $D[n] = \{d_{\lambda}: \lambda \in 1 + p^n\Z_p\}$, and $U^{-}[n] = \{u^{-}_r: r \in p^n\Z_p\}$. By standard Gauss elimination algorithm, we have
\begin{align}\label{eqn:GaussElimination}
    K_H[n] = U^{-}[n]D[n]U[n].
\end{align}

Since $X = G / \Gamma$ is compact, the injectivity radius of $X$ is positive. We use $\eta_X = p^{-\tilde{n}_0}$ be the injectivity radius. 

Let $\mu_G$ be the Haar measure on $G$ such that $\mu_G(K) = 1$ and $\mu_H$ be the Haar measure on $H$ such that $\mu_H(K_H) = 1$. Since $\Gamma$ is a lattice in $G$, $\mu_G$ induces a finite measure on $X = G / \Gamma$, we will denote this measure as $\mu_X$. We will use $\vol(X)$ to denote $\mu_X(X)$. 

Similarly, for a periodic $H$-orbit $Hg\Gamma$ in $X$, $g \Gamma g^{-1} \cap H$ is a lattice in $H$. The Haar measure $\mu_H$ induces a finite measure $\mu_{Hg\Gamma}$ on $H g\Gamma$. We will use $\vol(Hg\Gamma)$ to denote $\mu_{Hg\Gamma}(Hg \Gamma)$. 

\subsection{Lie Algebras}
Let $\LieG = \mathfrak{sl}_2(\Q_p) \oplus \mathfrak{sl}_2(\Q_p)$ and $\LieH = \{(x, x): s \in \mathfrak{sl}_2(\Q_p)\}$. Let $\|\cdot\|_p$ be the max-norm on $\mathrm{Mat}_2(\Q_p) \oplus \mathrm{Mat}_2(\Q_p)$ with respect to the standard basis. 

Let $\mathfrak{r} = \mathfrak{sl}_2(\Q_p) \oplus \{0\}$. We have $\LieG = \LieH \oplus \mathfrak{r}$. Note that $\mathfrak{r}$ is an ideal of $\LieG$. 

We will always use the following notation for elements $w \in \mathfrak{r}$: 
\begin{align*}
    w = \begin{pmatrix}
        w_{11} & w_{12}\\
        w_{21} & -w_{11}
    \end{pmatrix}
\end{align*}
where $w_{ij} \in \Q_p$. 

\subsection{\texorpdfstring{Constants and $*$-notations}{Constants and *-notations}}
Our convention on constant dependance is the same as the one in \cite{LM23}. For $A \ll B^*$, we mean there exist constants $C > 0$ and $\kappa$ depends at most on $(G, H, \Gamma)$ such that $A \leq C B^{\kappa}$. The $*$ main represents different $\kappa$ in one proof. For $A \asymp B$, we mean $A \ll B$ and $B \ll A$. For simplicity, if the constant depends at most on $(G, H, \Gamma)$, we will omit the dependance in the statement of the theorem. We emphasize here that the constants are allowed to depend on $p$. 

\subsection{\texorpdfstring{$p$-adic numbers and $\mathcal{S}$-adic numbers}{p-adic numbers and S-adic numbers}}Let $\Q_p$ be the field of $p$-adic numbers. We emphasize here that $|a|_p = p^{-v_p(a)}$ where $v_p$ is the $p$-adic valuation on $\Q_p$. We will always use $|\cdot|_p$ to denote the $p$-adic absolute value in this paper. We will use $|\cdot|$ to denote the standard absolute value on $\R$. 

We also record the following lemma. 
\begin{lemma}\label{lem:P-adicInterpolation}
    Let $a, b, c \in \Q_p$ with $\max\{|a|_p, |b|_p, |c|_p\} \geq 1$, then we have: 
    \begin{align*}
        |\{t \in \Z_p: |at^2 + bt + c|_p \leq p^n\}| \leq p^2 p^{\frac{1}{2}n}
    \end{align*}
    for all $n \in \Z$. 
\end{lemma}

\begin{proof}
    Let $f(t) = at^2 + bt + c$. 
    Suppose the conclusion does not hold, then there exists $t_i \in \Q_p$, $i = 1, 2, 3$ satisfying the following: 
    \begin{enumerate}
        \item $|f(t_i)|_p \leq p^n$;
        \item $|t_i - t_j|_p > p^{\frac{1}{2}n}$ for $i \neq j$. 
    \end{enumerate}
    Using Lagrange interpolation, we have
    \begin{align*}
        f(t) = \sum_{i} \prod_{j \neq i} \frac{(t - t_j)}{(t_i - t_j)}f(t_i).
    \end{align*}

    Therefore, the coefficient of $f$ has to be $< p^{-2(\frac{1}{2}n)} p^n = 1$, which leads to a contradiction. 
\end{proof}

Now we recall some basic notion on $\mathcal{S}$-adic number. Let $F$ be a number field. Let $\mathcal{S}$ be a finite set of places of $F$ containing all archimedean places.(c.f. \cite{PR94}) We will always use $\mathcal{S}_{\infty}$ to denote the set of all archimedean places. We will always assume that $F$ is a totally real field, that is, for all $v \in \mathcal{S}_{\infty}$, $F_v \cong \R$. 

For all $v \in \mathcal{S}$, there is a unique absolute value $|\cdot|_v$ such that its restriction to $\Q$ is one of $|\cdot|_p$ or the standard archimedean absolute value $|\cdot|_{\infty}$ on $\Q$. We will use $e_v$ to denote the ramification index of $F_v/\Q_p$. 

Let $F_{\mathcal{S}} = \prod_{v \in \mathcal{S}} F_{v}$ be the set of $\mathcal{S}$-adic numbers and $\mathcal{O}_{\mathcal{S}} = \{x \in F: |x|_{v} \leq 1 \text{ for all } v \notin \mathcal{S}\}$ be the set of $\mathcal{S}$-adic integers. Then diagonally embedded $\mathcal{O}_{\mathcal{S}}$ in $F_{\mathcal{S}}$ is a cocompact lattice. 

For an element $x = (x_v)_{v \in \mathcal{S}}$, we define its $\mathcal{S}$-absolute value as
\begin{align*}
    |x|_{\mathcal{S}} = \max_{v \in \mathcal{S}} \{|x_v|_v\}.
\end{align*}
We define its $\mathcal{S}$-height as
\begin{align*}
    \height_{\mathcal{S}}(x) = \prod_{v \in \mathcal{S}} |x_v|_v.
\end{align*}

Now we extend these notion to the space $F_{\mathcal{S}}^n = \prod_{v \in \mathcal{S}} F_v^n$. For an vector $\mathbf{x} = (x_i)_{i = 1}^n \in F_v^n$, we define its $v$-norm as $\|\mathbf{x}\|_v = \max_{i = 1, ..., n} |x_i|_v$ when $v$ is non-archimedean and $\|\mathbf{x}\|_v = (\sum_{i = 1}^n x_i^2)^{\frac{1}{2}}$ when $v$ is archimedean. Now for a $\mathcal{S}$-vector $\mathbf{x} = (\mathbf{x}_v)_{v \in \mathcal{S}}$ in $F_{\mathcal{S}}^n$, we define its $\mathcal{S}$-norm as
\begin{align*}
    \|\mathbf{x}\|_{\mathcal{S}} = \max_{v \in \mathcal{S}} \{\|\mathbf{x}_v\|_v\}.
\end{align*}
We also define its height as
\begin{align*}
    \height_{\mathcal{S}}(\mathbf{x}) = \prod_{v \in \mathcal{S}} \|\mathbf{x}_v\|_v.
\end{align*}
We remark here that there is a constant $c_{F, \mathcal{S}, n} > 0$ such that for all $\mathbf{x} \in \mathcal{O}_{\mathcal{S}}^n$, we have $\height_{\mathcal{S}}(\mathbf{x}) \geq c_{F, \mathcal{S}, n} > 0$. 

\subsection{\texorpdfstring{Reduction theory for $\mathcal{O}_{\mathcal{S}}$-lattices}{Reduction theory for OS-lattices}}\label{subsec:Reduction}
For all discrete $\mathcal{O}_{\mathcal{S}}$-module in $F_{\mathcal{S}}^n$, we record the following $\mathcal{S}$-adic version of Minkowski successive minima theorem proved in \cite{KST17}. 

\begin{theorem}
    Let $n \geq 1$ and let $\Gamma \subset F_{\mathcal{S}}^n$ be a discrete $\mathcal{O}_{\mathcal{S}}$-module with finite covolume. Let $\lambda_m(\Gamma)$ be its successive minima. Then 
    \begin{align*}
        \bigl(\lambda_1(\Gamma) ...\lambda_n(\Gamma)\bigr)^{\#\mathcal{S}} \asymp \vol(F_{\mathcal{S}}^n / \Gamma)
    \end{align*}
    where the implicit constants depends only on $F$, $\mathcal{S}$ and $n$. 
\end{theorem}

The following lemma is an $\mathcal{S}$-adic version of \cite[Chapter X, Lemma 4]{EE93}. 

\begin{lemma}\label{lem:SadicSiegel}
    Let $A \in \mathrm{M}_{m \times n}(\mathcal{O}_{\mathcal{S}})$. View $A$ as a map $A: F_{\mathcal{S}}^n \to F_{\mathcal{S}}^m$ by diagonally acting. Suppose $\|A\|_{\mathcal{S}} \leq T$. Then there exists $\mathcal{O}_{\mathcal{S}}$-basis $\xi_1$, ..., $\xi_s$ of $\ker A$ such that 
    \begin{align*}
        \|\xi_i\|_{\mathcal{S}} \ll_{F, \mathcal{S}} T^{3n}.
    \end{align*}
\end{lemma}

\begin{proof}
    The proof is exactly the same as the proof of \cite[Chapter X, Lemma 4]{EE93} if one replace the original Minkowski's second theorem by \cite[Theorem 1.2]{KST17} and \cite[Chapter X, Lemma 5]{EE93} by \cite[Lemma 3.5]{KST17}. 
\end{proof}

We also prove the following lemma similar to \cite[Lemma 13.1]{EMV09}. We call a subgroup $V$ of $F_{\mathcal{S}}^n$ is a $F$-subspace if $V = (V \cap \mathcal{O}_{\mathcal{S}}) \otimes_{\mathcal{O}_{\mathcal{S}}} F_{\mathcal{S}}$. 

\begin{lemma}\label{lem:AlmostSolution}
    Let $A \in \mathrm{M}_{m \times n}(\mathcal{O}_{\mathcal{S}})$. Let $\mathcal{S} = \mathcal{S}_1 \sqcup \mathcal{S}_2$ be a partition of $\mathcal{S}$. Suppose $\|A\|_{\mathcal{S}_1} \leq T$, $\|A\|_{\mathcal{S}_2} \leq C$. Suppose there exists $w \in F_{\mathcal{S}_1}^n$ such that $\|Aw\|_{\mathcal{S}_1} \leq \delta$, then there exists $w_0 \in \ker(A) \cap F_{\mathcal{S}_1}^n$ with 
    \begin{align*}
        \|w - w_0\|_{\mathcal{S}_1} \ll (CT)^{*} \delta.
    \end{align*}
\end{lemma}

\begin{proof}
    Note that $\ker(A)$ is a $F$-subspace of $F_{\mathcal{S}}^n$ and $\im(A)$ is a $F$-subspace of $F_{\mathcal{S}}^m$. By \cref{lem:SadicSiegel}, they have $\mathcal{O}_{\mathcal{S}}$-basis with $\mathcal{O}_{\mathcal{S}}$-norm $\ll C^{3n}T^{3n}$. Let $W = \ker(A)$ and $J = \im(A)$. Let $W^{\perp}$ denote the orthogonal complement of $W$. Since $F$ is a totally real field, we have $W \oplus W^{\perp} = F_{\mathcal{S}}^n$. Then using \cref{lem:SadicSiegel}, it contains basis with $\mathcal{O}_{\mathcal{S}}$-norm $\ll (CT)^*$. Let $B = A|_{W^{\perp}}: W^{\perp} \to J$. Then $B$ is an invertible matrix with entries in $F$, and each entry could be written as fraction of two elements in $\mathcal{O}_{\mathcal{S}}$ with $\mathcal{O}_{\mathcal{S}}$-absolute value $\ll (CT)^*$. 

    Now write $w = w_0 + w^{\perp}$ where $w_0 \in \ker(A) \otimes_{F} F_{\mathcal{S}_1}$ and $w^{\perp} \in W^{\perp} \otimes_{F} F_{\mathcal{S}_1}$. It suffices to estimate $\|w^{\perp}\|_{\mathcal{S}_1}$. 

    Note that $\|w^{\perp}\|_{\mathcal{S}_1} = \|B^{-1}Bw^{\perp}\|_{\mathcal{S}_1} \leq \|B^{-1}\|_{\mathcal{S}_1}\|Bw^{\perp}\|_{\mathcal{S}_1} \leq \|B^{-1}\|_{\mathcal{S}_1} \delta$, it suffices to estimate $\|B^{-1}\|_{\mathcal{S}_1}$. Note that every entry in $B^{-1}$ could also be written as fraction of two elements in $\mathcal{O}_{\mathcal{S}}$ with $\mathcal{O}_{\mathcal{S}}$-absolute value $\ll (CT)^*$. It suffices to bound each entry which holds due to \cref{lem:BoundHtNorm}. 
\end{proof}

\begin{lemma}\label{lem:BoundHtNorm}
    Let $x \in \mathcal{O}_{\mathcal{S}}$ be a nonzero element. Suppose $\|x\|_{\mathcal{S}} \leq C$, then $\|\frac{1}{x}\|_{\mathcal{S}} \leq C^{\#\mathcal{S} - 1}$.
\end{lemma}

\begin{proof}
    By product formula, we have
    \begin{align*}
        \prod_{v \in \mathcal{S}} |x|_v \geq 1.
    \end{align*}
    Therefore, 
    \begin{align*}
        \min_{v \in \mathcal{S}} |x|_v \cdot (\max_{v \in \mathcal{S}} |x|_v)^{\#\mathcal{S} - 1} \geq 1,
    \end{align*}
    which implies the bound on $\|\frac{1}{x}\|_{\mathcal{S}}$. 
\end{proof}

\subsection{An equivariant projection lemma}We prove an equivariant projection lemma similar to \cite[Lemma 13.2]{EMV09}. 

Let $G$ be a $p$-adic semisimple group and $S$ a closed semisimple subgroup of $G$. Suppose $G$ acts linearly on a $\Q_p$ linear space $V$ and there exists $0 \neq v_S \in V$ such that $\mathrm{Stab}(v_S) = S$. The map $g \mapsto g.v_S$ induce a map $\LieG \to V$. Since $S$ is semisimple, we could choose a $S$-invariant complement $W$ of the image of $\LieG$. We have the following lemma on local structure of $G$-orbits in $V$ near $v_S$. 

\begin{lemma}
    There exists a neighborhood $\mathcal{N}$ of $v_S$ such that the following holds. 
    
    Let $\Pi :\mathcal{N} \ni g.(v_S + w) \mapsto g.v_S$. $\Pi$ is a well-defined $G$-equivariant projection defined on $\mathcal{N}$.
\end{lemma}

\begin{proof}
    This is a direct conclusion of \cite[Proposition 4.1]{BGM19} and \cite{Lun75}. For more discussion on equivariant projection, see \cite[Theorem 4.5]{AHR20}. 
    We remark here that it could also be proved as in \cite[Section  13.4]{EMV09} using the language of analytic manifolds, c.f. \cite[Chapter IV]{Ser03}. 
\end{proof}

Now let $G = \SL_2(\Q_p) \times \SL_2(\Q_p)$, $S = H$, and $\Gamma$ a lattice in $G$. We have the following lemma analogous to \cite[Lemma 13.2]{EMV09}. 
\begin{lemma}\label{lem:EquiProj}
    There exists a constant $\bar{\kappa} > 0$ such that the following holds. 

    Let $v \in \mathcal{N}$ such that $\gamma.v = v$ for some $\gamma \in \Gamma$. Suppose $\|v - v_H\|_p \leq \|\gamma\|_p^{-\bar{\kappa}}$, then also $\gamma \Pi(v) = \Pi(v)$. 
\end{lemma}

\begin{proof}
    See \cite[Lemma 13.2]{EMV09}. 
\end{proof}

\subsection{Baker--Campbell--Hausdorff formula}

We collect the following lemmas on local structure of $G$ and exponential map. 

\begin{lemma}\label{lem:BCH}
    There exist absolute constant $M_0$ so that the following holds. Let $0 < \beta \leq \beta_0$, and let $w_1, w_2 \in B_{\mathfrak{r}}(0, p^{-M_0})$. There is $w \in \mathfrak{r}$ which satisfy
    \begin{align*}
        \|w\|_p = \|w_1 - w_2\|_p
    \end{align*}
    so that $\exp(w_1)\exp(-w_2) = \exp(w)$. 
\end{lemma}

\begin{proof}
    Note that $\exp(\mathfrak{r}) = \SL_2(\Q_p) \times \{e\}$, using Baker-Campbell-Hausdorff formula (c.f. \cite[Chapter II \S6.4]{Bou89} or \cite[Part I, Chapter IV, 8]{Ser06}), there exists $\bar{w} \in \mathfrak{r}$ such that
    \begin{align*}
        \exp(w_1)\exp(-w_2) = \exp(w_1 - w_2 + \bar{w}).
    \end{align*}

    We also have the following explicit expression of $\bar{w}$: 
    \begin{align*}
        \bar{w} = \sum_{n \geq 1} H_n(w_1, w_2) = \sum_{n} \sum_{\substack{r + s = n,\\ r\geq 1, s \geq 1
        }} H_{r, s}(w_1, w_2)
    \end{align*}
    where $H_{r, s} = H_{r, s}' + H_{r, s}''$ and $H_{r, s}'$ and $H_{r, s}''$ is of the following forms: 
    \begin{align}\label{eqn:BCHFormula1}
    (r + s)H_{r, s}' =
    \sum_{m \geq 1}\frac{(-1)^{m - 1}}{m} \sum_{\substack{r_1 + ... + r_m = r\\
    s_1 + ... + s_{m - 1} = s - 1\\
    r_1 + s_1 \geq 1\\ ...\\ r_{m - 1} + s_{m - 1} \geq 1}} \biggl(\biggl(\prod_{i = 1}^{m - 1} \frac{(\ad w_1)^{r_i}}{r_i!}\frac{(-\ad w_2)^{s_i}}{s_i!}\biggr)\frac{(\ad w_1)^{r_m}}{r_m!}\biggr)(-w_2);
    \end{align} \begin{align}\label{eqn:BCHFormula2}
    (r + s)H_{r, s}'' = 
    \sum_{m \geq 1} \frac{(-1)^{m - 1}}{m} \sum_{\substack{r_1 + ... + r_{m - 1} = r - 1\\
    s_1 + ... + s_{m - 1} = s\\
    r_1 + s_1 \geq 1\\...\\r_{m - 1} + s_{m - 1}\geq 1}} \biggl(\prod_{i = 1}^{m - 1} \frac{(\ad w_1)^{r_i}}{r_i!}\frac{(\ad -w_2)^{s_i}}{s_i!}\biggr)(w_1).
    \end{align}

    Using the estimate $v_p(n!) \leq \frac{n}{p - 1}$, (c.f. \cite[Part III, Chapter V, 4, Lemma 4]{Ser06}), the coefficient of $H_{r, s}$ is bounded by $p^{\frac{3(r + s)}{p - 1}}$. 

    Note that we have the following estimate for adjoint action: 
    \begin{align}\label{eqn:Estimate_ad}
    \begin{aligned}
        \|\ad(x)(y)\|_p ={}&  \|xy - yx\|_p\\
    ={}& \|xy - x^2 + x^2 -yx\|_p\\
    \leq{}& \|x\|_p\|x - y\|_p.
    \end{aligned}
    \end{align}

    Combining \cref{eqn:BCHFormula1}, \cref{eqn:BCHFormula2} and \cref{eqn:Estimate_ad}, we have: 
    \begin{align*}
        \|H_{r, s}\|_p \leq p^{\frac{3(r + s)}{p - 1}} p^{-M_0(r + s)}\|w_1 - w_2\|_p.
    \end{align*}
    Adding them together, we get
    \begin{align*}
        \|\bar{w}\|_p \leq \sum_{n \geq 1} n^2 p^{\frac{3n}{p - 1}} p^{-M_0n} \|w_1 - w_2\|_p.
    \end{align*}

    Letting $M_0$ large enough, we have
    \begin{align*}
        \|\bar{w}\|_p \leq p^{-1}\|w_1 - w_2\|_p.
    \end{align*}

    Letting $w = w_1 - w_2 + \bar{w}$, we have
    \begin{align*}
        \|w\|_p = \|w_1 - w_2\|_p.
    \end{align*}
    
\end{proof}

\begin{lemma}\label{lem:LocalProduct}
    There exists $\beta_0$ so that the following holds for all $0 < \beta \leq \beta_0$. Let $x \in X$and $w \in B_{\mathfrak{r}}(0, \beta)$. If there are $h, h' \in B_{\beta}^{H}$ so that $\exp(w')hx = h'\exp(w)x$, then 
    \begin{align*}
        h' = h\text{ and } w' = \Ad(h)w. 
    \end{align*}
    Moreover, we have $\|w'\|_p = \|w\|_p$. 
\end{lemma}

\begin{proof}
    The first statement follows from the fact that the map 
    \begin{align*}
        H \times \mathfrak{r} \to{}& G\\ 
        (h, w) \mapsto{}& h\exp(w)
    \end{align*}
    is a bi-analytic map near $(e, 0)$. See \cite[Chapter III \S4]{Bou89}. 

    The second statement follows from the fact that $h \in K_H$ preserves the norm. 
\end{proof}  

\subsection{\texorpdfstring{The set $\mathsf{E}_{\eta, N, \beta}$}{The set E}}

Let $\eta_0 = \frac{1}{p^2}\min\{\eta_X, \beta_0\} = p^{-n_0}$ where $\beta_0$ is from \cref{lem:BCH}. We fix a compact set $\mathfrak{D} \subset G$ such that 
\begin{enumerate}
    \item $G = \mathfrak{D} \Gamma$.
    \item $\mathfrak{D}$ is a disjoint union of $K[n_0]$-coset. 
\end{enumerate}

For all $0< \eta < \eta_0$ and $0 < \beta < \beta_0$, we define the set
\begin{align*}
    \mathsf{E}_{\eta, N, \beta} = K_{H, \beta} \cdot a_N \cdot \{u_r: |r|_p \leq \eta\}.
\end{align*}
We have $\mu_H(\mathsf{E}_{\eta, N, \beta}) \asymp \eta \beta^2 p^{2N}$. 

As in \cite{LM23}, $\mathsf{E}_{\eta, N, \beta}$ will be used only for $p^{-N/100} < \beta < \eta^2$. 

For $\eta, \beta, m > 0$, set
\begin{align*}
    \mathsf{Q}^H_{\eta, \beta, m} = \{u^{-}_s: |s|_p \leq p^{-m} \beta\} \cdot \{d_{\lambda}: |\lambda - 1|_p \leq \beta\} \cdot \{u_r: |r|_p \leq \eta\}.
\end{align*}
We write $\mathsf{Q}^H_{\beta, m}$ for $\mathsf{Q}^H_{\beta, \beta, m}$. 

The following lemma will be used in \cref{sec:MargulisFunctions}. 
\begin{lemma}\label{lem:PropertyOfQThickening}
    \begin{enumerate}
        \item The set $\mathsf{Q}^H_{\eta, \beta, m}$ is a subgroup of $K_H$.
        \item We have 
        \begin{align*}
             Q_{\beta, m}^H a_m u_r K_{H, \beta} \subset a_m u_r K_{H, \beta}.
        \end{align*}
    \end{enumerate}
\end{lemma}

\begin{proof}
    Note that for all $a, b, c, d$ satisfying $ad - bc = 1$ and $a \neq 0$, we have the following calculation:
    \begin{align*}
        \begin{pmatrix}
            a & b\\
            c & d
        \end{pmatrix} = \begin{pmatrix}
            1 & \\
            \frac{c}{a} & 1
        \end{pmatrix}\begin{pmatrix}
            a & \\
             & a^{-1}
        \end{pmatrix}\begin{pmatrix}
            1 & \frac{b}{a}\\
             & 1
        \end{pmatrix}.
    \end{align*}
    Therefore, we have
    \begin{align*}
        \mathsf{Q}^H_{\eta, \beta, m} = \biggl\{ \begin{pmatrix}
            a & b\\
            c & d
        \end{pmatrix}: |a - 1|_p \leq \beta\text{, } |d - 1|_p \leq \beta\text{, }|b|_p \leq \eta\text{, 
 and } |c|_p \leq \beta p^{-m}\biggr\}
    \end{align*}
    which shows $(\mathsf{Q}^H_{\eta, \beta, m})^{-1} \subset \mathsf{Q}^H_{\eta, \beta, m}$ and $\mathsf{Q}^H_{\eta, \beta, m} \cdot \mathsf{Q}^H_{\eta, \beta, m} \subset \mathsf{Q}^H_{\eta, \beta, m}$, which shows $\mathsf{Q}^H_{\eta, \beta, m}$ is a subgroup of $K_H$. 

    Property~(2) follows from the following calculation: 
    \begin{align*}
        u_s^- d_{\lambda} u_{r'} a_m u_r ={}& \begin{pmatrix}
            1 & \\
            s & 1
        \end{pmatrix}
        \begin{pmatrix}
            \lambda & \\
             & \lambda^{-1}
        \end{pmatrix}
        \begin{pmatrix}
            1 & r'\\
             & 1
        \end{pmatrix}
        \begin{pmatrix}
            p^{-m} & \\
             & p^m
        \end{pmatrix}
        \begin{pmatrix}
            1 & r\\
             & 1
        \end{pmatrix}\\
        ={}& \begin{pmatrix}
            p^{-m} & \\
             & p^m
        \end{pmatrix}
        \begin{pmatrix}
            1 & r(1 + p^{M_{\beta}}r'')\\
             & 1
        \end{pmatrix}
        \begin{pmatrix}
            1 - \lambda^2 r p^{-2m}s & - \lambda^4 r^2 p^{-2m}s\\
            p^{-2m}s & 1 + \lambda^2 r p^{-2m}s
        \end{pmatrix}\\
        {}&
        \begin{pmatrix}
            1 & \\
            p^{-2m}s & 1
        \end{pmatrix}
        \begin{pmatrix}
            \lambda & \\
             & \lambda^{-1}
        \end{pmatrix}
        \begin{pmatrix}
            1 & p^{m}r'\\
             & 1
        \end{pmatrix}.
    \end{align*}
\end{proof}

\subsection{A linear algebra lemma}
\begin{lemma}\label{lem:LinearAlgebra}
    Let $\frac{1}{3} < \alpha < 1$, $0 \neq w \in \LieG$, and $\lambda \in \Q_p$ with $|\lambda|_p > 1$. Then 
    \begin{align*}
        \int_{\Z_p} \|d_{\lambda} u_r.w\|_p^{-\alpha} \, dr \leq \frac{C_2 |\lambda|_p^{-\hat{\alpha}}}{p - p^{\alpha}} \|w\|_p^{-\alpha};
    \end{align*}
    where $C_2$ is an absolute constant and $\hat{\alpha} = \frac{1 - \alpha
    }{4}$. 
\end{lemma}
Let $m_{\alpha} \in \N$ defined by $\frac{C_2 p^{-\hat{\alpha}m_{\alpha}}}{p - p^{\alpha}} \alpha p^{-1}$. We will apply the lemma to $a_n$ where $n = \ell m_{\alpha}$. These imply
\begin{align*}
    \int_{\Z_p} \|a_{m_{\alpha}} u_r w\|_p^{-\alpha} \leq p^{-1} \|w\|_p^{-\alpha}.
\end{align*}

\subsection{Sobolev norm}
For functions in $L^2(X)$, let $\mathrm{Av}[m]$ be the averaging projection on $K[m]$-invariant functions, put $\mathrm{pr}[0] = \mathrm{Av}[0]$ and $\mathrm{pr}[m] = \mathrm{Av}[m] - \mathrm{Av}[m - 1]$ for $m \geq 1$. 

Let $f$ be a locally constant compactly supported function. Then the Soblev norm of degree $d$ is defined by 
\begin{align}
    \mathcal{S}_d(f)^2 = \sum_m p^{md} \|\mathrm{pr}[m]f\|_2^2. 
\end{align}
Roughly speaking, the Sobolev norm measures in what scale $f$ is locally constant on $X$. 

The Sobolev norm we defined here is a special case of the one defined on Adelic space in \cite{EMMV20}. We summarize the properties needed here and sketch a proof in \cref{sec:Sobolev}. For a throughout summary and proof, see \cite[Appendix A]{EMMV20}. 

\begin{proposition}\label{pro:PAdicSobolev}
    There exists $d_0$ such that for $d \geq d_0$, the Soblev norm $\mathcal{S}_d$ satisfies the following property. 

    \begin{enumerate}
    \item[(S1)] For all locally constant compactly supported function $f$, we have
    \begin{align*}
        \|f\|_{\infty} \ll \mathcal{S}_d(f).
    \end{align*} 
    \item[(S2)] For all $g \in G$, we have
    \begin{align*}
        \mathcal{S}_d(g \cdot f) \ll \|g\|^{4d} \mathcal{S}_d(f).
    \end{align*}
    \item[(S3)] For all $r \geq 0$ and $g \in K[r]$, we have 
    \begin{align*}
        \|g\cdot f - f\|_{\infty} \ll p^{-r} \mathcal{S}_d(f). 
    \end{align*}
    \item[(S4)] We have
    \begin{align*}
        \mathcal{S}_d(f_1f_2) \ll \mathcal{S}_d(f_1)\mathcal{S}_d(f_2).
    \end{align*}
\end{enumerate}
\end{proposition}    

To simplify notation, We fix some $d \geq d_0$ in the whole paper and write $\mathcal{S}(f) = \mathcal{S}_d(f)$.

\section{From large dimension to effective density}\label{sec:Venkatesh}
In this section, we will use the exponential decay of matrix coefficient of unitary representation of $H$ to prove \cref{thm:VenkateshTrick}, which is a $p$-adic analogue of \cite[Proposition 4.2]{LM23}. It says that the expansion translation of subset of $N$ which is foliated by $U$-orbits with dimension close to $2$ are equidistributed in $X$. 

The following theorem from \cite{Clo03} provides the estimate on the decay of correlation on $X$ we need. See also \cite{GMO06, EMMV20}. 

\begin{theorem}\label{mixing_1}
    There exists some $\kappa_2$ such that for all $h \in H$, for all locally constant functions $f_1, f_2 \in L^{2}_{0}(X)$, the matrix coefficient can be estimated as the follows
    \begin{align*}
        |\langle h.f_1, f_2\rangle| \leq {\dim \langle K.f_1 \rangle}^{\frac{1}{2}}{\dim \langle K.f_2 \rangle}^{\frac{1}{2}} \|f_1\|_{2}\|f_2\|_{2} \|h\|^{-\kappa_2}. 
    \end{align*}
    where $\langle K.f \rangle$ is the linear span of $K.f$. 
\end{theorem}

If $\Gamma$ is arithmetic group, $\kappa_1$ is absolute. 

Using the definition of the Soblev norm, we could get the following corollary. 

\begin{corollary}\label{cor:mixing_2}
    There exists $C_3$ and $d_0$ such that for all $d \geq d_0$, we have 
    \begin{align*}
        \biggl|\langle u_r.f_1, f_2\rangle - \int f_1 \int f_2\biggr| \leq C_3 (1 + |r|_p)^{-\kappa_2}\mathcal{S}_d(f_1)\mathcal{S}_d(f_2). 
    \end{align*}
\end{corollary}

\begin{proof}
    Note that if $f$ is $K[m]$-invariant, $\dim K.f \ll p^{m\dim X}$. 

    Therefore, we have 
    \begin{align*}
        \langle u_rf_1, f_2 \rangle &\leq \sum_m \sum_{m'} |\langle u_r\mathrm{pr}[m]f_1, \mathrm{pr}[m']f_2 \rangle|\\
        & \leq (1 + |r|_p)^{-\kappa_2}(\dim K.f_1)^{\frac{1}{2}}(\dim K.f_2)^{\frac{1}{2}}\|\mathrm{pr}[m]f_1\|_2\|\mathrm{pr}[m']f_2\|2 \\ &\leq (1 + |r|_p)^{-\kappa_2}\prod_{i = 1, 2} (\sum_m p^{\frac{m \dim X}{2}} \|\mathrm{pr}[m]f_i\|_2) \\ &\ll (1 + |r|_p)^{-\kappa_2}\mathcal{S}_{\dim X + 2}(f_1)\mathcal{S}_{\dim X + 2}(f_2). 
    \end{align*}
\end{proof}

Now we use Corollary \ref{cor:mixing_2} to prove the following statement. 

\begin{proposition}\label{pro:full-horospherical}
    There exists $\kappa_3 \gg \kappa_2$ so that the following holds. Let $0 < \eta < 1$, $\lambda \in \Q_p$ with $|\lambda|_p > 1$, and $x \in X$. Then for all $f \in \mathcal{S}(X)$, 
    \begin{align}
        \biggl|\int_{B_N(0,1)} f(d_{\lambda}n.x) \,dn - \int f \,d\mu_{X}\biggr| \leq C_4\mathcal{S}(f)|\lambda|_{p}^{-\kappa_3}
    \end{align}
    where $B_{N}(0, 1) = \biggl\{\biggl(\begin{bmatrix}
        1 & r\\
        0 & 1
    \end{bmatrix}, \begin{bmatrix}
        1 & s\\
        0 & 1
    \end{bmatrix}\biggr):r, s \in \mathbb{Z}_p\biggr\}$, and $C_4$ is an absolute constant with respect to volume of X and $\eta_0$, namely $C_4 \leq \mathrm{Vol}(X) \eta_X^{-*}$.  
\end{proposition}

\begin{proof}
    This statement is well known in many similar cases, see e.g \cite{BO12}, \cite[Proposition 4.1]{LM23}. We include the argument for convenience. 

    Let $\varphi^{+}$ be the indicator function on $B_N(0, 1)$. We could write $\varphi^{+} = \mathds{1}_{\mathbb{Z}_p^2} = \sum_{j = 0}^{p^{2n_0} - 1} \mathds{1}_{j + p^{n_0}\mathbb{Z}_p^2}$. Set $\varphi_j^{+} = \mathds{1}_{j + p^{n_0}\mathbb{Z}_p^2}$. Let $\kappa$ be some parameter we will optimize later. 

    Let $\varphi_j$ be an $|\lambda|_p^{-\kappa}$-thickening of $\varphi_j^{+}$ along the stable and central directions in $G$. Namely, $\varphi_j$ is the indicator function of $B_{N^{-}}^{|\lambda|_p^{-\kappa}}B_{D_G}^{|\lambda|_p^{-\kappa}} B_{N^{+}}^{\eta_X}. x$. 

    Note that $\mathcal{S}(\varphi_j) \ll \eta_X^{*}|\lambda|_p^{*\kappa}$

    By S3, we have
    \begin{align*}
        \biggl|\int_N f(d_{\lambda}n.x) \varphi_j^{+}(n) \,dn - \int_X f(d_{\lambda}y) \varphi_j(y) \,d\mu_X(y)\biggr| \ll \mathcal{S}(f) |\lambda|_p^{-\kappa}. 
    \end{align*} 

    Using Corollary \ref{cor:mixing_2}, we have, 
    \begin{align}
        \biggl|\int_X f(d_{\lambda}y) \varphi_j \,d\mu_X(y) - \int f \,d\mu_X \int \varphi_j \,d\mu_X\biggr| \ll{}& \mathcal{S}(f)\mathcal{S}(\varphi_j) |\lambda|_p^{-\kappa_1}\\
        \ll{}& \mathcal{S}(f) \eta_X^{-*} |\lambda|_p^{*\kappa}|\lambda|_p^{-\kappa_1}. 
    \end{align}
    The proposition follows by summing those $\eta_X^{-2}$ error terms and optimizing $\kappa$. 
\end{proof}

The following is a generalization of proposition 4.1 which replace the whole $B_{N}(0, 1)$ by certain subset with dimension close to $2$. This theorem is a $p$-adic analogue to \cite[Proposition 4.2]{LM23}. 

\begin{theorem}\label{thm:VenkateshTrick}
    There exists $\kappa_4$ and $\epsilon_0$ (both $\gg \kappa_2$) so that the following holds. Let $0 \leq \epsilon \leq \epsilon_0$ and $0 < b \leq 1/p^2$. Let $\rho$ be a probability measure on $\mathbb{Z}_p$ which satisfies 
    \begin{align}\label{eqn:VenkateshInputDimension}
        \rho(K) \leq Cb^{1 - \epsilon}
    \end{align}
    for all $K$ which is a ball of radius $b$ and a constant C. 
    Then, 
    \begin{align*}
    \biggl|\int_{\mathbb{Z}_p}\int_{\mathbb{Z}_p} f(d_{\lambda}u_r v_s.x ) \,dr \,d\rho(s) - \int f \,d\mu_X\biggr| \leq C_5 C \mathcal{S}(f) |\lambda|_{p}^{-\kappa_4}
    \end{align*}
    for all $b^{-\frac{1}{8}} \leq |\lambda|_p \leq b^{-\frac{1}{4}}$. 
    The constant $C_5 \ll \mathrm{Vol}(X) \eta_X^{-*}$. 
\end{theorem}

\begin{proof}
    Without loss of generality, we may assume $\int_X f dm_X = 0$. 

    Suppose $b = p^{-m_0}$, let $\mathbb{Z}_p = \bigsqcup_j a_j + p^{m_0}\mathbb{Z}_p$. Let $I_j = s_j + p^{m_0}\mathbb{Z}_p$, $c_j = \rho(I_j)$ for all $j$. Then $\sum_j c_j = 1$. 

    Let $B_j = \mathbb{Z}_p \times I_j$. Let $\varphi = \sum_j b^{-1} c_j \mathds{1}_{B_j}$. Using \cref{pro:PAdicSobolev}~(S3), we have  

    \begin{align*}
        &\biggl|\int_{\mathbb{Z}_p}\int_{\mathbb{Z}_p} f(d_{\lambda}u_rv_s.x) \,dr \,d\rho(s) - \sum_j c_j \int f(d_{\lambda}u_r v_{s_j}.x) \,dr\biggr|\\
        & \leq \sum_j \int_{I_j}\int |f(d_{\lambda}u_r v_s.x) - f(d_{\lambda}u_r v_{s_j}.x)| \,dr\, d\rho(s) \ll \mathcal{S}(f)b^{\frac{1}{2}}. 
    \end{align*}

    where we used the fact that $|\lambda^2|_p|s - s_j|_p \leq b^{-\frac{1}{2}}b = b^{\frac{1}{2}}$ in the last inequality. 

    Note that 
    \begin{align*}
        &\biggl|\sum_j c_j \int f(d_{\lambda}u_r v_{s_j}.x) \,dr - \int_N \varphi(n(r, s)) f(d_{\lambda}n(r, s).x) \,dr\,ds\biggr|\\ 
        \leq& \sum_j \int_{\mathbb{Z}_p} b^{-1} c_j \int_{I_j} \bigl|f(d_{\lambda} n(r, s_j). x) - f(d_{\lambda} n(r, s). x)\bigr| \,ds \,dr  \ll \mathcal{S}(f)b^{\frac{1}{2}}
    \end{align*}

    where we used the fact that $|\lambda^2|_p|s - s_j|_p \leq b^{\frac{1}{2}}$ again in the last inequality. 
    
    Therefore, it suffices to estimate

    \begin{align*}
        A = \int \varphi(n(r, s))f(d_{\lambda}n(r, s).x) \,dr \, ds. 
    \end{align*}

    Let $l \geq 2$ be a parameter which will be optimize later. Let $\tau = |\lambda|_p^{-(2 - \frac{1}{l})}$. Since $B_j = \mathbb{Z}_p \times I_j$, $u_rB_j = B_j$ for all $|r|_p \leq 1$. 

    Thus, 
    \begin{align*}
        A &= \int \varphi(n)f(d_{\lambda}n.x) \,dn \\
        &= \sum_j b^{-1}c_j\int_{B_j} f(d_{\lambda}n).x) \,dn\\
        &= \sum_j b^{-1}c_j\int_{B_j} f(d_{\lambda}u_r n).x) \,dn\\
        &= \frac{1}{\tau} \int_{|r|_p \leq \tau} \int \varphi(n)f(d_{\lambda}u_r n.x) \,dn \,dr.
    \end{align*}

By Cauchy-Schwarz inequality, we have

\begin{align*}
    |A|^2 \leq \int\biggl(\frac{1}{\tau} \int_{|r|_p \leq \tau} f(d_{\lambda}u_r n.x) \,dr\biggr)^2 \varphi(n) \,dn
\end{align*}

Since $c_j = \rho (I_j) \leq Cb^{1-\varepsilon}$, we have
\begin{align*}
    |A|^2 &\leq C b^{-\varepsilon} \int (\frac{1}{\tau} \int_{|r|_p \leq \tau} f(d_{\lambda}u_r n.x) \,dr)^2 \,dn\\
    &= \frac{1}{\tau^2} \int_{|r_1|_p \leq \tau} \int_{|r_2|_p \leq \tau} \int Cb^{-\varepsilon}\hat{f}_{r_1, r_2}(d_{\lambda}n. x)\,dn \,dr_1 \,dr_2. 
\end{align*}
where $\hat{f}_{r_1, r_2}(y) = f(d_{\lambda}u_{r_1}d_{\lambda^{-1}}.y)f(d_{\lambda}u_{r_2}d_{\lambda^{-1}}.y)$ for $|r_1|_p, |r_2|_p \leq \tau$. 

By S4, $\mathcal{S}(\hat{f}_{r_1, r_2}) \ll \mathcal{S}(f)^2(|\lambda|_p^2\tau)^{*} \ll \mathcal{S}(f)^2|\lambda|_p^{*/l}$. We choose $l \ll \frac{1}{\kappa_3}$ large enough so that 
\begin{align*}
    \mathcal{S}(\hat{f}_{r_1, r_2}) \ll \mathcal{S}(f)^2|\lambda|_p^{\kappa_3/2}.
\end{align*}

By proposition \ref{pro:full-horospherical}, we have 
\begin{align*}
    \biggl|b^{-\varepsilon} \int \hat{f}_{r_1, r_2}(d_{\lambda}n.x) \,dn\biggr| &= b^{-\varepsilon} \int_X \hat{f}_{r_1, r_2} \,d\mu_X + b^{-\varepsilon}\mathcal{O}(\mathcal{S}(\hat{f}_{r_1, r_2})|\lambda|_p^{-\kappa_3})\\
    &= b^{-\varepsilon} \int_X \hat{f}_{r_1, r_2} \,d\mu_X + b^{-\varepsilon}\mathcal{O}(\mathcal{S}(f)^2|\lambda|_p^{-\kappa_3/2}). 
\end{align*}

Since $b^{-\frac{1}{8}} \leq |\lambda|_p \leq b^{-\frac{1}{4}}$, if we choose $\varepsilon \leq \kappa_3/32$, then $b^{-\varepsilon} |\lambda|_p^{-\kappa_3/2} \leq b^{\kappa_3/32}$. 

Hence, 
\begin{align}
    \biggl|b^{-\varepsilon} \int \hat{f}_{r_1, r_2}(d_{\lambda}n.x) \,dn\biggr| = b^{-\varepsilon} \int_X \hat{f}_{r_1, r_2} \,d\mu_X + \mathcal{O}(\mathcal{S}(f)^2 b^{\kappa_3/32}). 
\end{align}

Using corollary \ref{cor:mixing_2}, we obtain the following bound if $|r_1 - r_2|_p > |\lambda|_p^{-2 + \frac{1}{2l}}$

\begin{align}
    \biggl|\int_X \hat{f}_{r_1, r_2} \,d\mu_X\biggr| \ll \mathcal{S}(f)^2 |\lambda|_p^{-\frac{\kappa_2}{2l}}. 
\end{align}

    Thus, we have 

    \begin{align*}
        |A|^2 \ll C\mathcal{S}(f)^2(b^{-\varepsilon}(|\lambda|_p^{-\frac{1}{2l}} + |\lambda|_p^{-\kappa_2/2l}) + b^{\kappa_3/32}). 
    \end{align*}

Note that $\kappa_3 \gg \kappa_2$, $l \ll \frac{1}{\kappa_3}$ if $\varepsilon \ll \kappa_3^2$, then altogether we finish the proof. 
    
\end{proof}

\section{A restricted projection theorem}\label{sec:proj}
In this section, we will prove the following proposition serves as an input of \cref{thm:VenkateshTrick}. This section is similar to \cite[Section 5]{LM23} while we change the restricted projection theorem to its analogue in $\Q_p^3$.

\begin{proposition}\label{pro:RestrictedProjMain}
    Let $0 < 10^{70}\epsilon < \alpha < 1$. Suppose there exists $x_1 \in X$ and $F \subset B_{\mathfrak{r}}(0, 1)$, containing $0$ such that
    \begin{align}\label{eqn:ProjDimensionEnergy}
        \sum_{w' \in F \backslash \{w\} } \|w' - w\|^{-\alpha} \leq D (\#F)^{1 + \epsilon} \text{ for all } w \in F,
    \end{align}
    for some $D \geq 1$. 

    Assume further that $\#F$ is large enough, depending explicitly on $\epsilon$. 

    Then there exists a finite set $I \subset \Z_p$, some $b_1 = p^{-l_1}$ with 
    \begin{align*}
        (\#F)^{-\frac{3 - \alpha + 5\epsilon}{3 - \alpha + 20 \epsilon}} \leq p^{-l_1} \leq (\#F)^{-\epsilon},
    \end{align*}
    and some $x_2 \in X$ so that the following statements hold. 
    \begin{enumerate}
        \item The set $I$ supports a probability measure $\rho$ which satisfies
        \begin{align*}
            \rho(J) \leq C_{\epsilon}' \cdot |J|^{\alpha - 30\epsilon}
        \end{align*}
        for all closed subgroup $J$ with $|J| \geq (\#F)^{-\frac{-15\epsilon}{3 - \alpha + 20\epsilon}}$, where $C_{\epsilon}'$ depends only on $\epsilon$ and $D$. 
        \item Let $N = \lceil\frac{l_1}{2}\rceil$. For all $s \in I$, we have
        \begin{align*}
            v_s.x_2 \in K[l_1] \cdot a_{N} \{u_r: r \in \Z_p\}.F.x_1.
        \end{align*}
    \end{enumerate}
\end{proposition}

\begin{remark}\label{rem:ProjConst}
    Here we discuss the estimate on the estimate on $C_{\epsilon}'$. We have
    \begin{align*}
        C_{\epsilon}' \ll D^{*} K^{K^{\frac{1}{\epsilon^2}}}
    \end{align*}
    for some absolute constant $K > 1$. 
    We remark here that in \cite{LM23}, using the restricted projection theorem in \cite{KOV21}, the corresponding constant has a better range $\ll \epsilon^{-*}$. 
\end{remark}

The proof of \cref{pro:RestrictedProjMain} is based on the following restricted projection proved in \cite{JL}. Its proof is based on a decoupling inequality for moment curve in $\Q_p^n$. 

\begin{theorem}\label{thm:ProjDecoupling2}
    Let $0 < \alpha < 1$, $0 < b_0 = p^{-l_0} < b_1 = p^{-l_1} < 1$ be three parameters. Let $E \subset B_{\mathfrak{r}}(0, b_1)$ be so that 
    \begin{align*}
        \frac{\#(E \cap B_{\mathfrak{r}}(w, b))}{\#E} \leq D'\cdot (\frac{b}{b_1})^{\alpha}
    \end{align*}
    for all $w \in \mathfrak{r}$ and all $b \geq b_0$, and some $D' \geq 1$. Let $0 < \epsilon < 10^{-70}$ and let $J$ be a ball in $\Z_p$. Let $\xi_r$ be the following map: 
    \begin{align*} 
        \xi_r(w) = \bigl(\Ad_{u_r}(w)\bigr)_{12} = w_{12} - 2w_{11} r - w_{21} r^2. 
    \end{align*}
    There exists $J' \subset J$ such that $|J'| \geq (1 - \frac{1}{p})|J|$ satisfying the following. Let $r \in J'$, then there exists a subset $E_r \subset E$ with 
    \begin{align*}
        \#E_r \geq (1 - \frac{1}{p}) \cdot (\#E)
    \end{align*}
    such that for all $w \in E_r$ and all $b \geq b_0$, we have
    \begin{align*}
        \frac{\#\{w' \in E: |\xi_r(w') - \xi_r(w)|_p \leq b\}}{\#E} \leq C_{\epsilon}\cdot(\frac{b}{b_1})^{\alpha - \epsilon}.
    \end{align*}
    where $C_{\epsilon}$ depends on $\epsilon$, $|J|$, $D'$ and could be chosen as in \cref{rem:ProjConst}. 
\end{theorem}

We also need the following version of \cite[Lemma 5.3]{LM23}. 

\begin{lemma}\label{lem:Bourgain}
    Let $F \subset B_{\mathfrak{r}}(0, 1)$ satisfying \cref{eqn:ProjDimensionEnergy}. Assuming $\#F$ is large enough depending on $\epsilon$. Then there exist $w_0 \in F$, $b_1 > 0$, with 
    \begin{align*}
        (\#F)^{-\frac{3 - \alpha + 5\epsilon}{3 - \alpha + 20 \epsilon}} \leq b_1 \leq (\#F)^{-\epsilon},
    \end{align*}
    and a subset $F' \subset B(w_0, b_1) \cap F$ so that the following holds. Let $w \in \mathfrak{r}$, and let $b \geq (\#F)^{-1}$. Then 
    \begin{align*}
        \frac{\#(F' \cap B(w, b))}{\#F'} \leq C' \cdot \biggl(\frac{b}{b_1}\biggr)^{\alpha - 20 \epsilon}.
    \end{align*}
    where $C' \ll_D \epsilon^{-*}$ with absolute implied constants. 
\end{lemma}

\begin{proof}
    Note that $\Z_p^3$ has a tree structure with $\deg = p^3$, replacing the dyadic cubes with balls in $\Z_p^3$, one could prove the lemma exactly the same as \cite[Appendix C]{LM23}. For a comprehensive construction of the subset of $\#F$ with a tree structure, see \cite[Section 2.2]{SG17}. See also \cite[Lemma 5.2]{BFLM11}, \cite[Section 2]{Bou10}, and \cite[Section A.3]{BG09}. We remark here the dependence of $\#F$ on $\epsilon$ could be chosen as
    \begin{align*}
        (\#F)^{\epsilon/2} > 4 \log_p(\#F).
    \end{align*}
\end{proof}

\begin{proof}[Proof of \cref{pro:RestrictedProjMain}]
    The proof is the same as \cite[Section 5]{LM23}. The strategy is straight forward. We first use \cref{lem:Bourgain} to replace $F$ with a local version of it. Then using \cref{thm:ProjDecoupling2}, we project the discretized dimension in $\mathfrak{r}$ to the direction of $\mathfrak{r} \cap \mathrm{Lie}(V)$. Finally, we use the action of $a_{N}$ to expand this subset to size $1$. 

    Assume $\#F$ is large enough depending on $\epsilon$ as the following:
    \begin{align*}
        (\#F)^{\epsilon/2} > \max\{4 \log_p(\#F), \beta_0^{-1}\}
    \end{align*}
    where $\beta_0$ is from \cref{lem:BCH}. 

    \noindent
    \textit{Step 1. Localizing the entropy. }Apply \cref{lem:Bourgain} with $F$ as in the proposition. Let $w_0 \in F$, $b_1 = p^{-l_1}$ and $F' \subset B_{\mathfrak{r}}(w_0, b_1) \cap F$ be given by that lemma; in particular, we have
    \begin{align*}
        (\#F)^{-\frac{3 - \alpha + 5 \epsilon}{3 - \alpha + 20 \epsilon}} \leq b_1 \leq (\#F)^{-\epsilon}.
    \end{align*}

    Now we defined $E$ to be subset of $B_{\mathfrak{r}}(0, b_1)$ to be as such points after changing the base point to $w_0$. Set 
    \begin{align*}
        E = \{w \in \mathfrak{r}: \exp(w) = \exp(w')\exp(-w_0) \text{ for some } w' \in F'\}.
    \end{align*}

    \begin{lemma}\label{lem:ProjChangeBasePoint}
    Let $E = \{w: w' \in F\}$ be as above. Then we have
    \begin{align}\label{eqn:DimensionBeforeShear}
        \frac{\#(E \cap B(w, b))}{\#E} \leq C' \cdot (\frac{b}{b_1})^{\alpha - 20 \epsilon}
    \end{align}
    for all $w \in \mathfrak{r}$ and $b \geq (\#F)^{-1}$ where $C'$ is from \cref{lem:Bourgain}. 
    \end{lemma}
    We will prove this lemma at the end of this section. 

    By the lemma, we have $E \subset B_{\mathfrak{r}}(0, b_1)$.
    \begin{lemma}\label{lem:ShearingInProj}
    There exists $r_0 \in \Z_p$ and a subset 
    \begin{align*}
        \hat{E} \subset \Ad_{u_{r_0}}E \cap \{w \in B_{\mathfrak{r}}(0, \eta): |w_{12}|_p \geq p^{-4}\|w\|_p\}
    \end{align*}
    so that $\#\hat{E} \geq \frac{\#E}{4}$. 
    \end{lemma}
    We will prove this lemma at the end of this section. 

    Let $x_2' = \exp(w_0). x_1$. By the \cref{lem:ShearingInProj}, we could assume
    \begin{align*}
        E \subset \{w \in B_{\mathfrak{r}}(0, \eta): |w_{12}|_p \geq p^{-4} \|w\|_p\}.
    \end{align*}

    Moreover, since $u_{r_0} \in K$, \cref{eqn:DimensionBeforeShear} holds for this new $E$. 

    \noindent
    \textit{Step 2. Estimates on size of elements. }Now let $N = \lceil\frac{l_1}{2}\rceil$. We have
    \begin{align*}
        a_{N} u_r. \exp(w). x_2' = a_N \exp(\Ad_{u_r} w) a_{-N} .a_{N} u_r. x_2'.
    \end{align*}

    Note that 
    \begin{align*}
        \Ad_{u_r}(w) = \begin{pmatrix}
            w_{11} + r w_{21} & w_{12} - 2r w_{11} - r^{2}w_{21}\\
            w_{21} & -w_{11} - rw_{21}
        \end{pmatrix}.
    \end{align*}
    If $|r|_p \leq p^{-5}$, we have
    \begin{align*}
        \bigl(\Ad_{u_r}(w)\bigr)_{12} \geq p^{-4}\|w\|_p.
    \end{align*}
    Now we use $a_N$ to expand those elements to size $1$ and close to $\LieN \cap \mathfrak{r}$. 
    We have the following calculation for $\Ad_{a_N u_r}(w)$: 
    \begin{align*}
        \Ad_{a_N u_r}(w) = \begin{pmatrix}
            w_{11} + r w_{21} & p^{2N}(w_{12} - 2r w_{11} - r^{2}w_{21})\\
            p^{-2N} w_{21} & -w_{11} - rw_{21}
        \end{pmatrix}.
    \end{align*}
    We have
    \begin{align}\label{eqn:ProjSizeOfElement}
        |\bigl(\Ad_{a_N u_r}(w)\bigr)_{11}|_p \leq{}& \|w\|_p;\\
        |\bigl(\Ad_{a_N u_r}(w)\bigr)_{21}|_p \leq{}& p^{-2N}\|w\|_p.
    \end{align}

    Let $J' \subset p^5 \Z_p$ be as in \cref{thm:ProjDecoupling2}. Fix one $r \in J'$. Let $I := \{p^{2N} \xi_r(w): w \in E_r\}$. We claim that $I$ satisfies the properties in \cref{pro:RestrictedProjMain}. 

    For proposition~(1), for all $b \geq p^{2N}\cdot (\#F)^{-1}$, we have
    \begin{align*}
        \rho(\{s' \in I: |s - s'|_p \leq b\}) ={}& \frac{\#\{w' \in E_r: |\xi_r(w') - \xi_r(w)|_p \leq p^{-2N} b\}}{\#E_r}\\
        \leq{}& C_{\epsilon} \biggl(\frac{p^{-2N} b}{p^{-l_1}}\biggr)^{\alpha - 30\epsilon}\\
        \leq{}& pC_{\epsilon} b^{\alpha - 30\epsilon}.
    \end{align*}

    Property~(2) follows directly from \cref{eqn:ProjSizeOfElement}. 
\end{proof}

\begin{proof}[Proof of \cref{lem:ProjChangeBasePoint}]
    Let $\eta$ small enough as in \cref{lem:BCH}. Let $f: B_{\mathfrak{r}}(0, \beta_0) \to B_{\mathfrak{r}}(0, \beta_0)$ by $f(w') = w$ where
    \begin{align*}
        \exp(w) = \exp(w')\exp(-w_0).
    \end{align*}

    By \cref{lem:BCH} and Baker--Campbell--Hausdorff formula, $f$ is bijection and $f^{-1}$ is analytic. 

    Therefore $\#E = \#f(F') = \#F'$ and 
    \begin{align*}
        \#(f(F') \cap B_{\mathfrak{r}}(\bar{w}, b)) = \#(F' \cap B_{\mathfrak{r}}(f^{-1}(\bar{w}), b))
    \end{align*}
    for all $b \leq \beta_0$.
\end{proof}

\begin{proof}[Proof of \cref{lem:ShearingInProj}]
    We prove by direct calculation. 

    Note that 
    \begin{align}
        \bigl(\Ad_{u_{r}} w\bigr)_{12} = -w_{21} r^2 - 2w_{11} r + w_{12}.
    \end{align}

    If 
    \begin{align*}
        \#\{w \in E: |w_{12}|_p \geq p^{-4} \|w\|_p\} \geq \frac{\#E}{4},
    \end{align*}
    then the claim holds for $r_0 = 0$. 

    Therefore, we assume $\#\hat{E} \geq \frac{3 \cdot (\#E)}{4}$ where $\hat{E} = \{w \in E: |w_{12}|_p < p^{-4}\|w\|_p\}$. If
    \begin{align*}
        \#\{w \in \hat{E}: |w_{21}|_p > p^{-1}\|w\|_p\} \geq \frac{\#E}{4},
    \end{align*}
    then the claim holds for $r = p^2$ and the set on the left side above. 

    If not, then we have
    \begin{align*}
        \#\{w \in \hat{E}: |w_{21}|_p < p^{-1}\|w\|_p\} \geq \frac{\#E}{2}.
    \end{align*}
    Taking $r = 1$ and the set on the left side above, we prove the claim. 
\end{proof}

\section{\texorpdfstring{Arithmetic lattices in $G$, Closed $H$-orbits and their volume}{Arithmetic lattices in G, Closed H-orbits and their volume}}\label{sec:ArithmeticLatticesClosedHOrbitVolume}
This section and \cref{sec:ClosingLemma} are the only two places in the paper where the arithmetic condition on $\Gamma$ is used. We will associate an arithmetic invariant to each periodic $H$-orbit in $X$ and compare it with the volume of periodic $H$-orbit in this section.  

\subsection{\texorpdfstring{Arithmetic lattices in $G$}{Arithmetic lattices in G}}\label{subsec:ArithmeticLattices}We first recall the definition of arithmetic lattice in this subsection. 

We begin with the case of irreducible arithmetic lattice. There is a number field $F$ and a $F$-simple algebraic group $\tilde{\mathbf{G}} \subset \SL_M$ satisfying the following. 
\begin{enumerate}
    \item For all archimedean places $v$ of $F$, $F_v \cong \R$ and $\tilde{\mathbf{G}}(F_v)$ is compact. 
    \item There is a non-archimedean place $v_0$ of $F$ such that $F_{v_0} \cong \Q_p$ and $\tilde{\mathbf{G}}(F_{v_0})$ is isogenous to $G = \SL_2(\Q_p) \times \SL_2(\Q_p)$. We use $\rho: \tilde{\mathbf{G}}(F_{v_0}) \to \SL_2(\Q_p) \times \SL_2(\Q_p)$ to denote this isogeny. 
    \item Let $\mathcal{S} = \{v: v | \infty\} \cup \{v_0\}$,  $\tilde{G}_{\mathcal{S}} = \prod_{v \in \mathcal{S}} \tilde{\mathbf{G}}(F_v)$, $\Gamma_{\mathcal{S}} = \tilde{\mathbf{G}}(\mathcal{O}_{\mathcal{S}})$. View $\Gamma_{\mathcal{S}}$ as diagonally embedded in $\tilde{G}_{\mathcal{S}}$, by Borel--Harish-Chandra theorem (c.f. \cite{PR94}), it is a lattice in $\tilde{G}_{\mathcal{S}}$. Let $\tilde{\rho}$ be the composition of $\rho$ and projection of $\tilde{G}_{\mathcal{S}}$ to $\tilde{\mathbf{G}}(F_{v_0})$, we have that $\Gamma$ is commensurable with $\tilde{\rho}(\Gamma_{\mathcal{S}})$. 
\end{enumerate}

Without loss of generality, we will always assume that $\Gamma = \tilde{\rho}(\Gamma_{\mathcal{S}})$. 

\begin{remark}
    In this case the group $\tilde{\mathbf{G}}$ could be chosen as $\mathrm{Res}_{K/F} \SL_2$ for some quadratic field extension $K/F$. 
\end{remark}

Now we give the definition we use when $\Gamma$ is reducible. There exists two number fields $F_1$, $F_2$ such that there exists $F_i$-simple groups $\tilde{\mathbf{G}}_i  \subset \SL_M$ for $i = 1, 2$ satisfying the following. 
\begin{enumerate}
    \item For all $i = 1, 2$ the following holds. For all archimedean places $v_i$ of $F_i$, $F_{i, v_i} \cong \R$ and $\tilde{\mathbf{G}}_i(F_{i, v_i})$ is compact. 
    \item For all $i = 1, 2$, $F_{i, v_i} \cong \Q_p$ and $\tilde{\mathbf{G}}_i(F_{i, v_i}) \cong \SL_2(\Q_p)$. Let $\rho_i: \tilde{\mathbf{G}}_i(F_{i, v_i}) \to \SL_2(\Q_p)$ be this isomorphism. 
    \item Let $\mathcal{S}_i$ be a finite set consisting of all archimedean places of $F_i$ and $v_i$. Let $\tilde{G}_{\mathcal{S}_i} = \prod_{v \in \mathcal{S}_i} \tilde{\mathbf{G}}(F_{i, v})$ and $\Gamma_{\mathcal{S}_i} = \tilde{\mathbf{G}}_i(\mathcal{O}_{\mathcal{S}_i})$. View $\Gamma_{\mathcal{S}_i}$ as diagonally embedded in $\tilde{G}_{\mathcal{S}_i}$. By Borel--Harish-Chandra theorem, it is a lattice in $\Gamma_{\mathcal{S}_i}$. Let $\tilde{\rho}_i$ be the composition of $\rho_i$ and the projection from $\Gamma_{\mathcal{S}_i}$ to $\tilde{\mathbf{G}}_i(F_{i, v_i})$. Let $\Gamma_i = \tilde{\rho}_i(\Gamma_{\mathcal{S}_i})$, we have $\Gamma$ commensurable with $\Gamma_i \times \Gamma_2$. 
\end{enumerate}

In this case we always assume without loss of generality that $\Gamma = \Gamma_1 \times \Gamma_2$. 

\begin{remark}\label{rem:RestrictionOfScalar}
    One could also describe lattices in $\SL_2(\Q_p)$ as the following. Let $\mathbf{G}$ be an absolutely almost simple group defined over a totally real number field $F$. Suppose $v_0$ is a place of $F$ such that $F_{v_0} \cong \Q_p$ and $\mathbf{G}(F_{v_0}) \cong \SL_2(\Q_p)$. Let $\mathcal{S} = \{v: v|\infty\text{ or } v = v_0\}$ and $\mathcal{S}' = \{v: v|\infty\text{ or } v|p\}$. Let $\bar{\mathbf{G}} = \mathrm{Res}_{F/\Q} \mathbf{G}$. Note that there is an isogeny $\bar{\rho}: \bar{\mathbf{G}}(\R) \times \bar{\mathbf{G}}(\Q_p) \to \prod_{v|\infty}\mathbf{G}(F_v) \times \prod_{v|p}\mathbf{G}(F_v) = \mathbf{G}(F_{\mathcal{S}'})$. Let $\rho$ be the composition of projection from $\mathbf{G}(F_{\mathcal{S}'})$ to $\mathbf{G}(F_{v_0})$ and the isomorphism from $\mathbf{G}(F_{v_0})$ to $\SL_2(\Q_p)$. Let 
    \begin{align*}
        \hat{\Gamma}_{\mathcal{S}'} = \bar{\mathbf{G}}(\Z[\frac{1}{p}]) \cap (\bar\rho)^{-1}(\prod_{v|\infty}\mathbf{G}(F_v) \times \mathbf{G}(F_{v_0}) \times \prod_{v \neq v_0}\mathbf{G}(\mathcal{O}_v)).
    \end{align*}
    Then a lattice $\Gamma$ in $\SL_2(\Q_p)$ is arithmetic if and only if it is commensurable with $\rho \circ \bar{\rho}(\hat{\Gamma}_{\mathcal{S}'})$. We will use this description in the proof of \cref{lem:AlmostClosedOrbit}. 
\end{remark}

\begin{remark}\label{rem:FieldOfDef}
    Note that if $\Gamma_1 \cap \Gamma_2$ is Zariski dense in $H$, we could assume $F_1 = F_2$ by passing to commensurable lattice. By taking Galois conjugate of $\tilde{\mathbf{G}}_i$, we could assume $v_1 = v_2$. In particular, if $X$ admits closed $H$-orbit, we have $\Gamma_1 \cap h\Gamma_2h^{-1}$ is Zariski dense in $H$ for some $h \in H$. 
\end{remark}

\begin{remark}
    In both cases, there is a finite index subgroup of $\Gamma$ which is torsion-free. Since we allow dependence on $\Gamma$, we will always assume $\Gamma$ is torsion-free by passing to finite index subgroup. 
\end{remark}

\subsection{\texorpdfstring{Closed $H$-orbit and its volume}{Closed H-orbit and its volume}}\label{subsec:Volume}In this subsection, we will connect two way of measure the complexity of a closed $H$-orbit. We will attach an arithmetic invariant, namely the discriminant, to each closed $H$-orbit. We will also determine its connection with the volume of a closed $H$-orbit. 

The material of this section is essentially from \cite[Section 17]{EMV09} and \cite[Section 2]{ELMV11}. The way we bound volume of closed $H$-orbit via discriminant is similar to the one in \cite[Section 2]{ELMV11}. We remark here that one could also use methods in \cite{EMMV20} to bound volume of closed $H$-orbit via discriminant. 

We first define the discriminant of a closed $H$-orbit. 

Note that if $\Gamma = \Gamma_1 \times \Gamma_2$ is reducible and $Hg \Gamma$ is closed, assuming $g = (e, g_0)$ without loss of generality, we have $\Gamma_1 \cap g_0 \Gamma_2 g_0^{-1}$ a lattice in $H$. By changing $\Gamma$ to $g \Gamma g^{-1}$, we could assume that $F_1 = F_2$. Therefore, once there is a closed $H$-orbit, we could always assume that $\Gamma$ comes from a $F$-group (not necessary $F$-simple), and $\#(\mathcal{S} \backslash \mathcal{S}_{\infty}) = 1$. 

\begin{remark}
    We will follow the convention that in a lemma/proposition/theorem, if the condition says there exists a closed $H$-orbit, then $\Gamma$ comes from a $F$-group (not necessary $F$-simple), and $\#(\mathcal{S} \backslash \mathcal{S}_{\infty}) = 1$. We also follow the convention that if the conclusion in a lemma/proposition/theorem says there exists a closed $H$-orbit, then from the condition in that lemma/proposition/theorem, one could get that $\Gamma$ comes from a $F$-group (not necessary $F$-simple), and $\#(\mathcal{S} \backslash \mathcal{S}_{\infty}) = 1$. 
\end{remark}

Let $V = (\wedge^3\LieG)^{\otimes 2}$. For all $g \in G$, pick basis $e_1, e_2, e_3$ of $\Ad(g^{-1})\LieH$, we define
\begin{align*}
    v_{Hg} = \frac{(e_1 \wedge e_2 \wedge e_3)^{\otimes 2}}{\det(B(e_i, e_i))} \in V.
\end{align*}
By the adjoint invariance of the Killing form $B$, we have that $v_{Hg}$ does not depends on the choice of basis and representative of $Hg$. 

Suppose $Hg\Gamma$ is closed, then $\Gamma$ is a lattice in $g^{-1}Hg$. Consider $\Lambda = \{\tilde{\gamma} \in \Gamma_{\mathcal{S}}: \rho(\tilde{\gamma}) \in g^{-1}Hg\}$ and let $\tilde{\mathbf{L}}$ be the Zariski closure of $\Lambda$, it is a $F$-group in $\tilde{\mathbf{G}}$ and $\rho(\tilde{\mathbf{L}}(F_v)) = g^{-1} H g$. 

Let $\tilde{L}_{\mathcal{S}} = \prod_{v \in \mathcal{S}} \tilde{\mathbf{L}}(F_v)$. Using Borel--Harish-Chandra theorem, We have that $\Gamma_{\mathcal{S}} \cap \tilde{L}_{\mathcal{S}}$ is a lattice in $\tilde{L}_{\mathcal{S}}$. 

Now we have $\dim_{F} \tilde{\mathbf{L}} = \dim_{\Q_p} H = 3$. Let $\tilde{V}_{F} = (\wedge^3 \tilde{\LieG})^{\otimes 2}$ and $\tilde{V}_{\mathcal{S}} = \prod_{v \in \mathcal{S}} \tilde{V}_F \otimes_{F} F_v$. Let $\tilde{\LieG}_{\Z} = \tilde{\LieG} \cap \mathfrak{sl}_M(\mathcal{O}_{\mathcal{S}})$ and $\tilde{V}_{\mathcal{O}_{\mathcal{S}}} = (\wedge^3 \tilde{\LieG}_{\mathcal{O}_{\mathcal{S}}})^{\otimes 2}$. Diagonally embedding $\tilde{V}_{\mathcal{O}_{\mathcal{S}}}$ into $\tilde{V}_{\mathcal{S}}$, we get a  discrete, cocompact  $\mathcal{O}_{\mathcal{S}}$-module in $\tilde{V}_{\mathcal{S}}$. 

Now we define the norm on $\tilde{\LieG}_{v} = \tilde{\LieG} \otimes_{F} F_v$. Since $\tilde{\mathbf{G}}(F_v)$ is compact for all non-archimedean places, the Killing form is negative definite. We define the norm on $\tilde{\LieG}_{v} = \tilde{\LieG} \otimes_{F} F_v$ by this Killing form. For non-archimedean places, we use the pullback norm via $\mathrm{d}\rho$. These norms induces norms and height on $\tilde{V}_{\mathcal{S}}$. 

Let $\tilde{\LieL} = \mathrm{Lie}(\tilde{\mathbf{L}})$. Pick a basis $e_1, e_2, e_3$ of $\tilde{\LieL}$, we define 
\begin{align*}
    \tilde{v}_{\tilde{\LieL}} = \frac{(e_1 \wedge e_2 \wedge e_3)^{\otimes 2}}{\det B(e_i, e_j)}.
\end{align*}
As for $v_{Hg}$, it is independent of the choice of basis. Moreover, we have that $(\wedge^3\mathrm{d}\rho)^{\otimes 2}(\tilde{v}_{\tilde{\LieL}}) = v_{Hg}$ and $\tilde{v}_{\tilde{\LieL}}$ only depends on $Hg$. Therefore, we will also use $\tilde{v}_{Hg}$ to denote $\tilde{v}_{\tilde{\LieL}}$. 

Now let's consider the diagonally embedded $v_{\tilde{\LieL}}$ in $\tilde{V}_{\mathcal{S}}$. Since $\tilde{\LieL}$ is a $F$-subspace, there exists $x \in \mathcal{O}_{\mathcal{S}}$ such that $x v_{\tilde{\LieL}} \in \tilde{V}_{\mathcal{O}_{\mathcal{S}}}$. We define the discriminant of $H g\Gamma$ via
\begin{align*}
    \disc(Hg \Gamma) = \min\{\height(x): x \in \mathcal{O}_{\mathcal{S}}\}.
\end{align*}
This is well-defined since $\height(\mathcal{O}_{\mathcal{S}}) \subset \Z$ and $\Gamma_{\mathcal{S}}$ preserves $\tilde{V}_{\mathcal{O}_{\mathcal{S}}}$. 

Note that if we could find $\mathcal{O}_{\mathcal{S}}$-basis $\{e_i\}_{i = 1, 2, 3}$ of $\tilde{\LieL}$ with $\max_{v \in \mathcal{S}}\|e_i\| \leq T$, then 
\begin{align*}
    \disc(H g \Gamma) \leq \height(\det(B(e_i, e_j))) \leq |\det(B(e_i, e_j))|_{\mathcal{S}}^{\#\mathcal{S}} \ll T^{3 \#\mathcal{S}}.
\end{align*}

As in \cite[Section 2]{ELMV11}, we prove a separation estimate on closed $H$-orbit (c.f.\cite[Proposition 2.3, 2.4]{ELMV11}). 

\begin{lemma}\label{lem:SeparationOfClosedHOrbit}
    Let $H g_1 \Gamma$ and $H g_2 \Gamma$ be two closed $H$-orbits in $G / \Gamma$ with $N_{G}(H) g_1 \neq N_{G}(H) g_2$. Suppose $\|\Ad(g_1)\|_{op}, \|\Ad(g_2)\|_{op} \leq R$. Let $D_1 = \disc(Hg_1\Gamma)$, $D_2 = \disc(Hg_1\Gamma)$, then there exists $C_6 > 0$ depending only on $(G, H, \Gamma)$ such that for all  $p^{-N} \leq C_6^{-1} R^{-12} D_1^{-1} D_2^{-1}$, we have
    \begin{align*}
        g_1 \notin K[N]g_2.
    \end{align*}
\end{lemma}

\begin{proof}
    We first show that $v_{Hg_1} \neq v_{Hg_2}$. If not, then $g_1^{-1} g_2$ fixes $v_{H}$, which shows that $g_2 g_1^{-1} \in N_G(H)$, contradict to the condition that $N_G(H)g_1 \neq N_G(H)g_2$. 

    Since $\mathrm{d}\rho$ is an isomorphism between Lie algebra, we have $\tilde{v}_{Hg_1} \neq \tilde{v}_{Hg_2}$. Pick $x_i \in \mathcal{O}_{\mathcal{S}}$ such that $x_i \tilde{v}_{Hg_i} \in \tilde{\LieG}_{\mathcal{O}_{\mathcal{S}}}$. Note that $x_i$ is up to $\mathcal{O}_{\mathcal{S}}^{\times}$. Then we have $x_1 x_2 \tilde{v}_{Hg_i} \in \tilde{\LieG}_{\mathcal{O}_{\mathcal{S}}}$ which implies that $\height(x_1 x_2 \tilde{v}_{Hg_1} - x_1 x_2 \tilde{v}_{Hg_2}) \geq 1$. Hence, 
    \begin{align*}
        \height(x_1) \height(x_2) \prod_{v \in \mathcal{S}} \|\tilde{v}_{Hg_1} - \tilde{v}_{Hg_2}\|_{v} \geq 1. 
    \end{align*}

    Since for all $v | \infty$, the $\R$-group $\tilde{\mathbf{G}}(F_v)$ is compact, the Killing form $B$ is negative definite on $\tilde{\LieG}_{F_v}$ for all $v | \infty$. Therefore
    \begin{align*}
        \|\tilde{v}_{Hg_i}\|_v \asymp 1
    \end{align*}
    for all $v| \infty$. 

    Therefore, we have
    \begin{align*}
        \|\tilde{v}_{Hg_1} - \tilde{v}_{Hg_2}\|_{v_0} \gg D_1^{-1}D_2^{-1}.
    \end{align*}

    Reduce to $(\wedge^3\LieG)^{\otimes 2}$, we have
    \begin{align*}
        \|v_{Hg_1} - v_{Hg_2}\|_p \gg D_1^{-1}D_2^{-1}.
    \end{align*}

    Note that we have
    \begin{align*}
        \|v_{Hg_1} - v_{Hg_2}\|_p ={}& \|(\wedge^3\Ad)^{\otimes 2}(g_1^{-1}) v_H - (\wedge^3\Ad)^{\otimes 2}(g_2^{-1}) v_H\|_p\\
        \leq{}& R^{12}\|v_H\|_p \|\Id - (\wedge^3\Ad)^{\otimes 2}(g_1g_2^{-1})\|_{op}.
    \end{align*}

    Since $\Ad$ is an algebraic representation, there exists $C_6$ depending only on $(G, H, \Gamma)$ such that $g_1 g_2^{-1} \notin K[N]$ for all $N$ such that $p^{-N} \leq C_6^{-1} R^{-12} D_1^{-1} D_2^{-1}$. 
\end{proof}

Recall we fix a compact set $\mathfrak{D} \subset G$ such that 
\begin{enumerate}
    \item $G = \mathfrak{D} \Gamma$.
    \item $\mathfrak{D}$ is a disjoint union of $K_{\eta_0}$-coset. 
\end{enumerate}

\begin{lemma}\label{lem:VolumeDisc}
    For all closed orbit $H g \Gamma$ in $G/\Gamma$, we have
    \begin{align*}
        \vol(H g \Gamma) \ll \disc(H g \Gamma)^{6}.
    \end{align*}
\end{lemma}

\begin{proof}
    Pick a disjoint $K_{\eta_0}$-covering of $G/\Gamma$. Then we have $Hg \Gamma = \sqcup_{i \in I} Hg \Gamma \cap K_{\eta_0}g_i \Gamma$ where $K_{\eta_0}g_i \subset \mathfrak{D}$, $I$ is a finite set. Using the local structure, we have $Hg \Gamma \cap K_{\eta_0}g_i \Gamma = \sqcup_{j \in J_i} K_{H, \eta_0} g_{i, j} \Gamma$ where $K_{H, \eta_0} g_{i, j}  \subset K_{\eta_0}g_i$ and $H g_{i, j} \neq H g_{i, j'}$ for $j \neq j'$. Since $|N_G(H): H| = 2$, one could pick a subset $J_i' \subset J_i$ with $\#J_i' \geq \frac{1}{2} J_i$ and for all $g_{i, j} \neq g_{i, j'}$ such that $j, j' \in J_i'$, we have
    \begin{align*}
        N_G(H) g_{i, j} \neq N_G(H) g_{i, j'}.
    \end{align*}

    Let $D = \disc(Hg\Gamma)$. Let $R > 0$ such that for all $g \in \mathfrak{D}$, $\|\Ad_g\|_{op} \leq R$. Pick $N_D$ such that $\frac{1}{\lfloor CR^{12}\rfloor + 1}D^{-2} \leq p^{-N_D} \leq \frac{1}{CR^{12}}D^{-2}$. Using \cref{lem:SeparationOfClosedHOrbit}, we have $K[N_D]H g_{i, j'} \cap K[N_D]H g_{i, j'} = \emptyset$. Therefore, 
    \begin{align*}
        K[N_D]H g_{i, j'}\Gamma \cap K[N_D]H g_{i, j'}\Gamma = \emptyset.
    \end{align*}

    Hence we could get the following estimate on $\#J_i'$: 
    \begin{align*}
        \#J_i' p^{-3N_D} \eta_0^3 \ll \eta_0^6
    \end{align*}
    which imples
    \begin{align*}
        \#J_i' \ll \eta_0^3 D^6.
    \end{align*}

    Therefore, we have
    \begin{align*}
        \vol(Hg \Gamma) \ll \sum_{i \in I} \sum_{j \in J_i'} \vol(K_{H, \eta_0} g_{i, j} \Gamma) \ll \eta_0^3 \eta_0^{-6} \eta_0^3 D^{6}
        \ll D^{6}.
    \end{align*}
\end{proof}

The following lemma is an analogue to \cite[Lemma 6.2]{LM23}. 

\begin{lemma}\label{lem:ClosedOrbitVolume}
    There exists $C_7$ and $\kappa_5$ depends on $\mathfrak{D}$ and $\Gamma$ such that the following holds. Let $\gamma_1$ and $\gamma_2$ be two non-commuting elements in $\Gamma$. If $g \in \mathfrak{D}$ satisfies $\gamma_i g^{-1} v_H = g^{-1} v_H$, then $H g \Gamma$ is a periodic orbit such that: 
    \begin{align}\label{eqn:VolEstimate}
        \vol(H g \Gamma) \leq C_7 \max \{\|\gamma_1\|_p, \|\gamma_2\|_p\}^{\kappa_5}.
    \end{align}
\end{lemma}

\begin{proof}
    We first show that $H g\Gamma$ is a closed orbit. 

    Since $\gamma_i g^{-1} v_{H} = g^{-1} v_{H}$, we have that $g \gamma_i g^{-1} \in \mathrm{Stab}_{G}(v_{H}) = N_{G}(H)$. Let $L'$ be Zariski closure of $\langle g\gamma_1g^{-1}, g\gamma_2 g^{-1} \rangle$ in $G$. We claim that we could assume without loss of generality that $L' = H$. 

    If $g \gamma_i g^{-1} \in H$ for $i = 1, 2$, then let $\Lambda = \langle g \gamma_1 g^{-1}, g \gamma_2 g^{-1} \rangle \leqslant H$. Since $\Lambda$ is an infinite non-commutative discrete subgroup of $H$, it is Zariski dense in $H$. In this case $L' = H$. 

    If not, let $\Lambda = \langle g \gamma_1 g^{-1}, g \gamma_2 g^{-1} \rangle \leqslant N_G(H)$. We could assume without loss of generality that $g \gamma_1 g^{-1} \in N_G(H) \backslash H$ and $g \gamma_2 g^{-1} \in H$. In fact, if $g \gamma_i g^{-1} \in N_G(H) \backslash H$ for all $i = 1, 2$, then we could replace $\gamma_2$ by $\gamma_2\gamma_1$ since $|N_G(H): H| = 2$. Note that this only changes the exponent the right side by a factor $2$ of \cref{eqn:VolEstimate}. Now let $\Lambda_1 = \langle g \gamma_1^2 g^{-1}, g \gamma_2 g^{-1} \rangle$, this is a discrete, torsion free subgroup of $H$. We claim that $\Lambda_1$ is Zariski dense in $H$. It suffices to show that $\Lambda_1$ is noncommutative. Suppose it is commutative, by Ihara theorem (c.f. \cite{Ser03}), we know that $\Lambda_1 \cong \Z$. Let $\gamma'$ be a generator of $g^{-1}\Lambda_1 g$, we have that $\Lambda = \langle g \gamma_1 g^{-1}, g \gamma' g^{-1}\rangle$ and $\Lambda_1 = \langle g \gamma_1^2 g^{-1}, g \gamma' g^{-1}\rangle$, hence $\Lambda / \Lambda_1 \cong \Z / 2\Z$. This implies $\Lambda \cong \Z$, or $\Lambda \cong \Z \times \Z/ 2\Z$ or $\Lambda \cong \Z \rtimes \Z/ 2\Z$. The first two cases lead to a contradiction since $\Lambda$ is non-commutative. The last case lead to a contradiction since $\Lambda$ is torsion-free. 

    Therefore, we could assume without loss of generality that the Zariski closure of $\langle g \gamma_1 g^{-1}, g \gamma_2 g^{-1} \rangle$ is $H$. 

    If $\Gamma = \Gamma_1 \times \Gamma_2$ is a reducible lattice, letting $g= (g^{(1)}, g^{(2)})$, we have $g^{(1)} \Gamma_1 (g^{(1)})^{-1} \cap g^{(2)} \Gamma_2 (g^{(2)})^{-1}$ contains $\langle g \gamma_1 g^{-1}, g \gamma_2 g^{-1} \rangle$, which is Zariski dense in $H$. By \cref{rem:FieldOfDef}, we could pass to commensurable lattice and assume $\Gamma$ comes from a $F$-group, and $\#(\mathcal{S} \backslash \mathcal{S}_{\infty}) = 1$. 

    Let $L$ be the Zariski closure of $\langle \gamma_1, \gamma_2 \rangle$ in $G$. By the above discussion, we could assume that $L = g^{-1}H g$. 

    Now let $\tilde{\gamma}_{i} \in \Gamma_{\mathcal{S}}$ such that $\rho(\tilde{\gamma}_i) = \gamma_i$. Let $\tilde{\mathbf{L}}$ be the Zariski closure of $\langle \tilde{\gamma}_{1}, \tilde{\gamma}_{2} \rangle$ in $\tilde{\mathbf{G}}$, then $\tilde{\mathbf{L}}$ is a $F$-subgroup of $\tilde{\mathbf{G}}$. It is semisimple and $\rho(\tilde{\mathbf{L}}(F_{\mathcal{S}})) = g^{-1} H g$. By Borel--Harish-Chandra theorem, we have that $\tilde{\mathbf{L}}(F_{\mathcal{S}}) \cap \tilde{\mathbf{G}}(\mathcal{O}_{\mathcal{S}})$ is a lattice in $\tilde{\mathbf{L}}(F_{\mathcal{S}})$. 

    Therefore, $L\Gamma$ is a periodic orbit and $H g \Gamma$ is a periodic orbit. 
    \vspace{2mm}
    
    Now we prove the volume estimate \cref{eqn:VolEstimate}. 

    Let $\tilde{\LieL} \subset \tilde{\LieG}$ be the Lie algebra of $\tilde{\mathbf{L}}$. This is a $F$-subspace of $\tilde{\LieG}$. By \cref{lem:VolumeDisc}, it suffices to find $\mathcal{O}_{\mathcal{S}}$-basis of $\tilde{\LieL}$ with $\mathcal{O}_{\mathcal{S}}$-norm bounded by $\max\{\|\gamma_1\|_p, \|\gamma_2\|_p\}^*$. 

    Let $\Phi$ be the vectors in $\tilde{V}_{F}$ fixed by $\tilde{\mathbf{L}}$. Since $\langle \tilde{\gamma}_1, \tilde{\gamma}_2 \rangle$ is Zariski dense in $\tilde{\mathbf{L}}$, it contains a $\mathcal{O}_{\mathcal{S}}$-basis with norm bounded by $\max\{\|\tilde{\gamma_1}\|_{\mathcal{S}}, \|\tilde{\gamma_2}\|_{\mathcal{S}}\}^*$ by \cref{lem:SadicSiegel}. Note that $\tilde{\LieL} = \{x \in \tilde{\LieG}: x.q = 0 \text{ for all } q \in \Phi\}$, using \cref{lem:SadicSiegel} again, there exists $\mathcal{O}_{\mathcal{S}}$ basis of $\tilde{\LieL}$ with norm $\ll \max\{\|\tilde{\gamma_1}\|_{\mathcal{S}}, \|\tilde{\gamma_2}\|_{\mathcal{S}}\}^*$. 
    
    Now since $\rho$ is a $\Q_p$-algebraic representation, we could bound $\max\{\|\tilde{\gamma_1}\|_{v_0}, \|\tilde{\gamma_2}\|_{v_0}\}$ via power of $\max\{\|\gamma_1\|_{p}, \|\gamma_2\|_{p}\}$. For $v | \infty$, since $\tilde{\mathbf{G}}(F_v)$ is compact, there exists $C_7$ such that $\|\tilde{\gamma_i}\|_v \leq C_7$ for $i = 1, 2$. This proves the lemma. 

\end{proof}

\section{An effective closing lemma}\label{sec:ClosingLemma}

Recall that
\begin{align*}
    \mathsf{E}_{\eta, N, \beta} = K_{H, \beta} \cdot a_N \cdot \{u_r: |r|_p \leq \eta\}.
\end{align*}
We will use $\mathsf{E}_{N}$ to denote $\mathsf{E}_{1, N, \beta}$. 

Let $x \in X$ and $N > 0$, for every $z \in \mathsf{E}_{N}. x$, put
\begin{align*}
    I_{N}(z) = \{\omega \in \mathfrak{r}: 0 < \|\omega\|_p < \eta_0, \exp(\omega)z \in \mathsf{E}_{N}\}.
\end{align*}

We define the function $f_{N, \alpha}: \mathsf{E}_{N}.x \to [2, \infty)$ as the following: 
\begin{align*}
    f_{N, \alpha}(z) = \begin{cases}
        \sum_{\omega \in I_N(z)} \|\omega\|_p^{-\alpha} & \text{ if } I_N(z) \neq \emptyset\\
        \eta_0^{-\alpha}  & \text{ otherwise }
    \end{cases}.
\end{align*}

The main proposition of this section is the following analogue of \cite[Proposition 6.1]{LM23}. 
\begin{proposition}\label{pro:MainClosingLemma}
    There exists $D_0$ (which depends explicitly on $\Gamma$) satisfying the following. Let $D \geq D_0 + 1$, and let $x_0 \in X$. Then for all $N \gg_{X} 1$, at least one of the following holds. 
    \begin{enumerate}
        \item There exists $J \subset \Z_p$ with $|\Z_p \backslash J| \leq p^{-2N}$ such that for all $r \in J$, let $x_r  = a_{4N} u_r x_0$, we have
        \begin{enumerate}
            \item $h \mapsto h.x_r$ is injective over $\mathsf{E}_N$. 
            \item For all $z \in \mathsf{E}_N.x_r$, we have 
            \begin{align*}
                f_{N, \alpha}(z) \leq p^{DN}
            \end{align*}
            for all $0 < \alpha < 1$. 
        \end{enumerate}
        \item There is $x' \in X$ such that $H.x'$ is periodic with 
        \begin{align*}
            \mathrm{vol}(H.x') \leq p^{D_0N}\text{ and } x' \in K[(D - D_0)N]x_0.
        \end{align*}
    \end{enumerate}
\end{proposition}

As in \cite[Section 6]{LM23}, we first give the following lemma similar to \cite[Lemma 6.3]{LM23}. 

\begin{lemma}\label{lem:AlmostClosedOrbit}
There exists $C_8$, $\kappa_6$ and $\kappa_7$ depends on $\Gamma$ such that the following holds. Let $\gamma_1$ and $\gamma_2$ be two non-commuting elements. Let $N > 0$ be a positive integer such that 
\begin{align*}
    p^{-N} \leq C_8^{-1}\bigl(\max\{\|\gamma_1\|_p, \|\gamma_2\|_p\}\bigr)^{-\kappa_6}.
\end{align*}

Suppose there exists $g \in \mathfrak{D}$ such that $\gamma_i g^{-1} v_H = \epsilon_i g^{-1} v_H$ for $i = 1, 2$ and $\epsilon_i \in K[N]$. Then, there exists $g' \in G$ such that 
\begin{align*}
    \|g' - g^{-1}\|_p \leq C_8 p^{-N} \bigl(\max\{\|\gamma_1\|_p, \|\gamma_2\|_p\}\bigr)^{\kappa_7} 
\end{align*}
and $\gamma_i g' v_H =  g' v_H$ for $i = 1, 2$. 
\end{lemma}

\begin{proof}
    The proof is essentially the same as \cite[Lemma 6.3]{LM23}. 

    Let $\mathbf{L} = \rho^{-1}(g^{-1} H g)$ be the $\Q_p$-subgroup of $\tilde{\mathbf{G}}(F_{v_0})$ in the case where $\Gamma$ is irreducible or the $\Q_p$-subgroup of $\tilde{\mathbf{G}}_1\bigl((F_1)_{v_1}\bigr) \times \tilde{\mathbf{G}}_2\bigl((F_2)_{v_2}\bigr)$ in the case where $\Gamma$ is reducible. Let $w_0$ be a unit length vector in $\wedge^3 \LieL$. 

    Pick $\tilde{\gamma}_i \in \Gamma_{\mathcal{S}}$ such that $\rho(\tilde{\gamma}_i) = \gamma_i$ for $i = 1, 2$. Note that $\tilde{\gamma}_i$ are matrices with entries in $\mathcal{O}_{\mathcal{S}}$. Moreover, since $\rho$ is an algebraic representation and $F/\Q$ is a finite field extension, we have $\|\tilde{\gamma}_{i}\|_{v_0} \ll \|\gamma_i\|_p^*$ in the case where $\Gamma$ is irreducible and $\max\{\|\tilde{\gamma}_{i}\|_{v_1}, \|\tilde{\gamma}_{i}\|_{v_2} \}\ll \|\gamma_i\|_p^*$ in the case where $\Gamma$ is reducible. Also, there exists $C'$ depending on $\Gamma$ such that $\|\tilde{\gamma}_i\|_{v} \leq C'$ for $v | \infty$ since $\mathbf{G}(F_v)$ is compact for all archimedean place. 

    We first deal with the case where $\Gamma$ is an irreducible lattice. Consider the map 
    \begin{align*}
        A = (\tilde{\gamma}_1 - \Id) \oplus (\tilde{\gamma}_2 - \Id): \wedge^3 \text{Lie}(\mathbf{G}(F_{v_0})) \to \wedge^3 \text{Lie}(\mathbf{G}(F_{v_0})) \oplus \wedge^3 \text{Lie}(\mathbf{G}(F_{v_0})).
    \end{align*}
    We have that $\|Aw_0\|_{v_0} \leq p^{-N}$. By \cref{lem:AlmostSolution}, there exists $w' \in \wedge^3 \text{Lie}(\mathbf{G}(F_{v_0}))$ such that $A w' = 0$ and 
    \begin{align*}
        \|w' - w_0\|_{v_0} \leq{}& C p^{-N} \max\{ \|\tilde{\gamma_1}\|_{\mathcal{S}}, \|\tilde{\gamma_2}\|_{\mathcal{S}}\}^*\\
        \leq{}& C (C')^{*} \eta_0^{-1} p^{-N} \max\{ \|\gamma_1\|_{p}, \|\gamma_2\|_{p}\}^{\kappa'}
    \end{align*}
    for some absolute constant $C$ and $\kappa'$. 

    Therefore, $\tilde{\gamma}_i w' = w'$. By \cref{lem:EquiProj}, there exists $\bar{C}_8$, $\bar{\kappa}_6$ such that if
    \begin{align*}
        \|w' - w_0\|_{v_0} \leq \bar{C}_8^{-1} \max\{ \|\gamma_1\|_p, \|\gamma_2\|_p\}^{-\bar{\kappa}_6},
    \end{align*}
    there exists $\tilde{g} \in \mathbf{G}(F_{v_0})$ such that $\|\tilde{g} - \Id\| \leq C'' \|w' - w_0\|$ and 
    \begin{align*}
        \tilde{\gamma}_i \tilde{g} w_0 = \tilde{g} w_0
    \end{align*}
    for $i = 1, 2$. 

    Now let 
    \begin{align*}
        p^{-N} \leq (\bar{C}_8C)^{-1} (C')^{-*}(\max\{\|\gamma_1\|_p, \|\gamma_2\|_p\})^{-\bar{\kappa}_6 - \kappa'}.
    \end{align*}
    Then there exists $\tilde{g}$ such that
    \begin{align*}
        \|\tilde{g} - \Id\|_{v_0} \leq C'' C (C')^* \eta_0^{-1} p^{-N} \max\{\|\gamma_1\|_p, \|\gamma_2\|_p\}^*
    \end{align*}
    and 
    \begin{align*}
        \tilde{\gamma}_i \tilde{g} w_0 = \tilde{g} w_0
    \end{align*}
    for $i = 1, 2$.     

    Now we deal with the case where $\Gamma$ is a reducible lattice. Note that the above discussion holds if the arithmetic lattice $\Gamma$ satisfies $F_1 = F_2$ in the definition in \cref{sec:ArithmeticLatticesClosedHOrbitVolume}. Therefore, it suffices to deal with the case where $F_1 \neq F_2$. In this case, we will use the description of arithmetic lattices in \cref{rem:RestrictionOfScalar}. We will use $\bar{\mathbf{G}}_i$ to denote $\mathrm{Res}_{F_i/\Q} \tilde{\mathbf{G}}_i$ and we will assume that $\Gamma = \rho_1(\bar{\rho_1}(\hat{\Gamma}_{\mathcal{S}_1'})) \times \rho_2(\bar{\rho_2}(\hat{\Gamma}_{\mathcal{S}_2'}))$ by passing to finite index subgroup. Let $\bar{\mathbf{G}} = \bar{\mathbf{G}}_1 \times \bar{\mathbf{G}}_2$ and let $\rho$ and $\bar{\rho}$ be the corresponding product homomorphism. 

    Let $\bar{\mathbf{L}} = \bar{\rho}^{-1}(\rho^{-1}(g^{-1}Hg) \cap \mathbf{G}_1(F_{1, v_1}) \times \mathbf{G}_2(F_{2,v_2}))$, then $\bar{\mathbf{L}}$ is a $\Q_p$-subgroup of $\bar{\mathbf{G}}(\Q_p)$. Let $\bar{w}_0$ be a unit length vector in $\wedge^3 \bar{\mathfrak{l}}$. Let $\bar{\gamma_i} = (\rho \circ \bar{\rho})^{-1}(\gamma_i)$. We have that the component corresponding to $F_{i, v_i'}$ when $v_i' \neq v_i$ is bounded since it lies in $\tilde{\mathbf{G}}_i(\mathcal{O}_{v_i'})$ which is compact. Therefore, there exists $C'$ such that 
    \begin{align*}
        \|\bar{\gamma_i} - \Id\|_{\infty, p} \leq C'\max\{\|\gamma_1\|_p, \|\gamma_2\|_p\}^*
    \end{align*}
    for $i = 1, 2$. 

    Consider the map
    \begin{align*}
        A = (\bar{\gamma}_1 - \Id) \oplus (\bar{\gamma}_2 - \Id): \wedge^3 \text{Lie}(\mathbf{G}(\Q_p)) \to \wedge^3 \text{Lie}(\mathbf{G}(\Q_p)) \oplus \wedge^3 \text{Lie}(\mathbf{G}(\Q_p)).
    \end{align*}
    We have that $\|A w_0\| \leq p^{-N}$. Using similar argument as in the previous case, we get the same result as the previous case. 

    Combining those two cases and using the fact that $\rho$ is an algebraic map, we get 
    \begin{align*}
        \|\rho(\tilde{g}) - \Id\|_p \leq C_8' p^{-N} \max\{\|\gamma_1\|_p, \|\gamma_2\|_p\}^*
    \end{align*}
    for some $C_8' > 0$. Also, since $g \in \mathfrak{D}$, we have
    \begin{align*}
        \|\rho(\tilde{g})g^{-1} - g^{-1}\|_p \leq C_8 p^{-N} \max\{\|\gamma_1\|_p, \|\gamma_2\|_p\}^*
    \end{align*}
    for some $C_8 > 0$. 

    Let $g' = \rho(\tilde{g})g^{-1}$, we have
    \begin{align*}
        \gamma_i g' v_H = g' v_H. 
    \end{align*}
\end{proof}

\begin{remark}
    In the proof of the case where $\Gamma = \Gamma_1 \times \Gamma_2$, we actually showed that if the condition of the lemma is satisfied, then using \cref{lem:ClosedOrbitVolume}, $\Gamma_1$ and $\Gamma_2$ has to be defined over the same field and they are commensurable up to conjugation. However, a priori, we don't know this information on $\Gamma$. Therefore we need to use \cref{rem:RestrictionOfScalar} to overcome this difficulty. 
\end{remark}

We also give an estimate on $\#I_N(z)$ in the following lemma. 

\begin{lemma}\label{lem:ClosingLemmaNumberOfSheet}
    Let $x \in X$. Then for every $z \in \mathsf{E}_{N}.x$, we have
    \begin{align*}
        \# I_N(z) \ll p^{2N}.
    \end{align*}
\end{lemma}

\begin{proof}
    For all $z \in \mathsf{E}_{N}.x$ and $\omega \in I_N(z)$, since $K[n] \leqslant K[m]$ for $n > m$, we have that:
    \begin{align*}
        K_H[N_1] \exp(\omega)z \subset \mathsf{E}_{N}.x.
    \end{align*}

    Note that by local product structure, we have
    \begin{align*}
        K_H[N_1] \exp(w_1).z \cap K_H[N_1] \exp(w_2).z = \emptyset
    \end{align*}
    for $w_1 \neq w_2$, $w_1, w_2 \in I_N(z)$.

    Then since $m_H(\mathsf{E}_N) \ll p^{2N}$, $m_H(K_H[N_1]) \asymp p^{-3N_1}$, we get the final conclusion.  
\end{proof}

\begin{proof}[Proof of \cref{pro:MainClosingLemma}]
    Write $x_0 = g_0 \Gamma$ where $g_0 \in \mathfrak{D}$. 
    
    We start by assuming case (1) does not hold. Then there is a subset $E \subset \Z_p$ with measure $|E| > p^{-2N}$ such that for all $r \in E$, letting $h_r = a_{4N}u_r$, at least one of the following holds for $h_r.x_0$: 
    \begin{itemize}
        \item either the map $\mathsf{h} \mapsto \mathsf{h} h_r x_0$ is not injective on $\mathsf{E}_{N}$, 
        \item or there exists $z \in \mathsf{E}_{N}. h_r.x_0$ so that $f_{N, \alpha}(z) > p^{DN}$.  
    \end{itemize}

    \emph{Step 1. Finding lattice elements.}

    Let's start from the former situation. This implies that $\mathsf{h}_1 h_r x_0 = \mathsf{h}_2 h_r x_0$ for some $\mathsf{h}_1, \mathsf{h}_2 \in \mathsf{E}_{N}$, $\mathsf{h}_1 \neq \mathsf{h}_2$.  
    Let $\mathsf{s}_r = \mathsf{h}_2^{-1} \mathsf{h}_1$, we have that 
    \begin{align}\label{eqn:LatticeElementConstruction1}
        h_r^{-1} \mathsf{s}_r h_r = g_0 \gamma_r g_0^{-1}
    \end{align}
    where $\gamma_r \neq e$. 

    Now we focus on the former situation. 

    If $f_{N, \alpha}(z) > p^{DN}$, by taking $N$ large enough such that $p^N > \eta_0^{-1}$, we have $I_N(z) \neq \emptyset$ and 
    \begin{align*}
        \sum_{\omega \in I_N(z)} \|\omega\|_p^{-\alpha} > p^{DN}.
    \end{align*}

    Since $\#I_N(z) \ll p^{N}$, there must exists one $\omega \in I_{N}(z)$ with 
    \begin{align*}
        0 < \|\omega\|_p \ll p^{(-D + 1)N}.
    \end{align*}

    By taking $N$ large enough depending only on $G$, we could assume that: 
    \begin{align*}
        0 < \|\omega\|_p \leq p^{(-D + 2)N}.
    \end{align*}

    Now we have $\mathsf{h}_1, \mathsf{h}_2 \in \mathsf{E}_{N}$ with $\mathsf{h}_1 \neq \mathsf{h}_2$ such that $\exp(\omega)\mathsf{h}_1 h_r.x_0 = \mathsf{h}_2 h_r.x_0$. Thus, 
    \begin{align*}
        \exp(\omega_r) h_r^{-1} \mathsf{s}_{r} h_r. x_0 = x_0
    \end{align*}
    where $\mathsf{s}_{r} = h_2^{-1} h_1$ and $\omega_r = \Ad(h_r^{-1} \mathsf{h}_2^{-1})\omega$. 

    We have $\|\omega_r\|_p \ll p^{12N} \|\omega\|_p \leq p^{(-D + 12)N}$. 

    By letting $N$ large enough, we have: 
    \begin{align*}
        0 < \|\omega_r\|_p \leq p^{(-D + 13)N}
    \end{align*}

    Using $x_0 = g_0 \Gamma$ for $g_0 \in \mathfrak{D}$, we have the following: 
    \begin{align}\label{eqn:LatticeElementConstruction2}
        \exp(\omega_r) h_r^{-1} \mathsf{s}_{r} h_r = g_0 \gamma_r g_0^{-1}
    \end{align}
    where $e \neq \mathsf{s}_r \in H$ and $e \neq \gamma_r \in \Gamma$.

    \emph{Step 2. Some properties of the elements $\gamma_r$. }
    
    We claim that those lattice elements we picked in step 1 has the following two properties: 
    \begin{enumerate}
        \item $\|\gamma_{r}\|_p \leq p^{11N}$;
        \item There are $\gg p^{\frac{1}{2}N}$ distinct elements in $\{\gamma_r: r \in E\}$. 
    \end{enumerate}

    Property (1) follows from direct calculation using the definition of $\gamma_r$. In former situation, we have: 
    \begin{align*}
        \gamma_r = g_0^{-1} h_r^{-1} \mathsf{s}_r h_r g_0. 
    \end{align*}
    Therefore, we have the following estimate: 
    \begin{align*}
        \|\gamma_r^{\pm1}\|_p \ll{}& \|h_r^{-1} \mathsf{s}_r^{\pm1} h_r\|\\
        \ll& p^{10N}
    \end{align*}
    where the implicit constant depends only on $\mathfrak{D}$. By enlarging $N$ depends on this constant, we get that 
    \begin{align*}
        \|\gamma_r^{\pm1}\|_p \leq p^{11N}.
    \end{align*}

    In the latter situation, we have similar estimate: 
    \begin{align*}
        \|\gamma_r^{\pm1}\|_p \ll{}& \|\exp(\omega_r) h_r^{-1} \mathsf{s}_r^{\pm1} h_r\|_p\\
        \ll{}& p^{10N}.
    \end{align*}
    The implicit constant depends only on $\mathfrak{D}$. By enlarging $N$ depends on this constant, we get that 
    \begin{align*}
        \|\gamma_r^{\pm}\|_p \leq p^{11N}.
    \end{align*}

    Now we show that property (2) holds. 

    Let $M_1 > 0$ such that $g\gamma g^{-1} \cap \pm K[M_1] = \emptyset$ for all $\gamma \in \Gamma \backslash \{e\}$ and $g \in \mathfrak{D}$. This is possible since $\Gamma$ is torsion-free. 

    Write $\mathsf{s}_{r} = \begin{pmatrix}
        a_1 & a_2\\
        a_3 & a_4
    \end{pmatrix} \in H$ where $|a_i|_p \leq p^{3N}$. By \cref{eqn:LatticeElementConstruction1} or \cref{eqn:LatticeElementConstruction2}, we have: 
    \begin{align*}
        h_r^{-1} \mathsf{s}_r h_r \notin \pm K[M_1].
    \end{align*}

    Therefore, we have: 
    \begin{align*}
        \|h_r^{-1} \mathsf{s}_r h_r \pm \Id\|_p = \biggl\|u_{-r}\begin{pmatrix}
            a_1 & p^{8N}a_2\\
            p^{-8N}a_3 & a_4
        \end{pmatrix}u_{r} \pm \Id \biggr\|_p \geq p^{-M_1}
    \end{align*}
    which implies that 
    \begin{align*}
        \max\{p^{8N}|a_3|_p, |a_1 - 1|_p, |a_4 - 1|_p\} \geq p^{-M_1},
    \end{align*}
    \begin{align*}
        \max\{p^{8N}|a_3|_p, |a_1 + 1|_p, |a_4 + 1|_p\} \geq p^{-M_1}.
    \end{align*}

    Suppose $p^{8N}|a_3|_p < p^{-M_1}$, then $|a_1a_4 - 1|_p = |a_2 a_3|_p \leq p^{-5N - M_1}$. If $|a_1 - 1|_p \neq |a_4 - 1|_p$ or $|a_1 + 1|_p \neq |a_4 + 1|_p$, we have $|a_1 - a_4|_p \geq p^{-M_1}$. Otherwise $|a_1 - 1|_p |a_4 + 1|_p = |a_1 a_4 - 1 + a_1 - a_4|_p \geq p^{-2M_1}$. Since $|a_1a_4 - 1|_p \leq p^{-5N - M_1}$, we have $|a_1 - a_4|_p \geq p^{-2M_1}$. Putting those discussions together, we have 
    \begin{align}\label{eqn:LatticeElementEstimateAux}
        \max\{p^{8N}|a_3|_p, |a_1 - a_4|_p\} \geq p^{-2M_1}.
    \end{align}

    For all $r \in E$, let $J_r = \{r' \in E: \gamma_{r'} = \gamma_{r}\}$. We claim the following estimate on $|J_r|$: 
    \begin{align*}
        |J_r| \leq p^{-\frac{5}{2}N}.
    \end{align*}

    In former situation, we have
    \begin{align*}
        h_r^{-1} \mathsf{s}_r h_r = h_{r'}^{-1} \mathsf{s}_{r'} h_{r'}.
    \end{align*}

    We have: 
    \begin{align*}
        \mathsf{s}_r = h_r h_{r'}^{-1} \mathsf{s}_{r'} h_{r'} h_{r}^{-1}.
    \end{align*}

    Let $\tau = p^{-8N}(r - r')$, we have
    \begin{align*}
        \mathsf{s}_r = u_{\tau} \mathsf{s}_{r'} u_{-\tau}.
    \end{align*}

    In latter situation, we have
    \begin{align*}
        h_r^{-1} \mathsf{s}_r h_r ={}& \exp(-\omega_r)g_0 \gamma_r g_0^{-1}\\
        ={}& \exp(-\omega_r) \exp(\omega_{r'})h_{r'}^{-1} \mathsf{s}_{r'} h_{r'}\\
        ={}& \exp(\omega_{rr'}) h_{r'}^{-1} \mathsf{s}_{r'} h_{r'}.
    \end{align*}

    We have
    \begin{align*}
        \mathsf{s}_{r} = \exp(\hat{\omega}_{rr'})u_{\tau} \mathsf{s}_{r'}u_{-\tau}
    \end{align*}
    where $\|\hat{\omega}_{rr'}\|_p = \|\Ad(h_r) \omega_{rr'}\|_p \ll p^{(-D + 21)N}$. 

    Since $\|\mathsf{s}_r\|_p \leq p^{3N}$, we have 
    \begin{align*}
        \biggl\| \begin{pmatrix}
            a_1 + a_3\tau & a_2 + (a_4 - a_1)\tau -a_3 \tau^2\\
            a_3 & a_4 - a_3\tau
        \end{pmatrix}\biggr\| \leq p^{3N}.
    \end{align*}

    Now we have
    \begin{align*}
        |p^{8N} a_2 + (a_4 - a_1)(r - r') - a_3 p^{-8N}(r - r')^2| \leq p^{-5N}.
    \end{align*}

    By \cref{eqn:LatticeElementEstimateAux}, at least one of the coefficient of this polynomial has size $\geq p^{-M_1}$. By \cref{lem:P-adicInterpolation}, we have that $|J_r| \ll p^{\frac{-5}{2}N}$. Hence there are $\gg p^{\frac{1}{2}N}$ distinct elements in $\{\gamma_r: r \in \Z_p\}$. 
    
    \emph{Step 3. Zariski closure of the group generated by $\{\gamma_r: r \in \Z_p\}$. }

    \emph{Case 1.} The family $\{\gamma_r: r \in \Z_p\}$ is commutative. 

    We will show this case does not occur. 

    Recall that since $\Gamma$ is a discrete subgroup of $G = \SL_2(\Q_p) \times \SL_2(\Q_p)$, it is cocompact and hence contains no unipotent element. We will get a contradiction using this fact. 

    Let $\mathbf{L}$ be the Zariski closure of $\langle \gamma_r: r \in \Z_p \rangle$. Since $\langle \gamma_r: r \in \Z_p \rangle$ is commutative, so is $\mathbf{L}$. By \cite[Theorem 16.13]{Mil17}, $\mathbf{L} = \mathbf{T}\mathbf{V}$ where $\mathbf{T}$ is a (possibly finite) algebraic subgroup of a torus, $\mathbf{V}$ is a unipotent group. 

    If both $\mathbf{T}$ and $\mathbf{V}$ are non-central, We claim that they have to belong to different factor. For all $\gamma_r = (\gamma_{r, 1}, \gamma_{r, 2})$, let $\gamma_r^s = (\gamma_{r, 1}^s, \gamma_{r, 2}^s)$, $\gamma_r^u = (\gamma_{r, 1}^u, \gamma_{r, 2}^u)$ be the corresponding Jordan decomposition of $\gamma_r$. Note that in $\SL_2(\Q_p)$, an element is either semisimple or unipotent or product of $-\Id_2$ and a unipotent element and its centralizer has to lie in the same class. Therefore, $\mathbf{T}$ and $\mathbf{V}$ has to be in different factor. Without loss of generality, we assume that $\mathbf{T}$ is in the first factor. 

    However, for every torus $T \subset \SL_2(\Q_p)$, we have
    \begin{align}\label{eqn:CountingInTorus}
        \#B_T(e, R) \cap \Gamma \ll \log_p R,
    \end{align}
    where the constant is absolute. Combining the facts that $\#\{\gamma_r: r\in \Z_p\} \geq p^{\frac{1}{2}N}$ and $\|\gamma_r\| \leq p^{11N}$, there must exists $\gamma_{r} \neq \gamma_{r'}$ such that $\gamma_{r, 1} = \gamma_{r', 1}$. Therefore, $(e, \gamma_{r, 2}^{-1}\gamma_{r', 2})$ is a nontrivial unipotent element in $\Gamma$, which leads to a contradiction. 

    Now we have that either one of  $\mathbf{T}$ and $\mathbf{V}$ is central, then $\mathbf{L} = \mathbf{T}'C_{G}$ where $\mathbf{T}'$ is an algebraic subgroup of a torus since there is no unipotent element in $\Gamma$. We get a contradiction by \cref{eqn:CountingInTorus}. 

    \emph{Case 2.} There are $r, r' \in \Z_p$ such that $\gamma_r$ and $\gamma_{r'}$ do not commute. 

    Let $v_H$ be as in \cref{lem:AlmostClosedOrbit}. Then since $\exp(w_r) h_r^{1} \mathsf{s}_r h_r = g_0 \gamma_r g_0^{-1}$,
    \begin{align*}
        \gamma_r g_0^{-1} v_H ={}& g_0^{-1} \exp(w_r) h_r^{-1} \mathsf{s}_r h_r g_0 g_0^{-1} v_H\\
        ={}& \exp(\Ad_{g_0^{-1}} w_r) g_0^{-1} v_H
    \end{align*}
    where $\|\Ad_{g_0^{-1}} w_r\|_p \ll p^{(-D + 21)N}$. Similar statement holds for $r'$. 

    Therefore, if $D$ is large enough, then we could conclude that there exists $g_1 \in G$ with
    \begin{align*}
        \|g_1 - g_0\|_p \leq C_8 p^{(-D + 21+ 11\kappa_7)N}
    \end{align*}
    so that $\gamma_i g_1^{-1} v_H = g_1^{-1} v_H$. 

    Using \cref{lem:ClosedOrbitVolume}, $Hg_1\Gamma$ is a closed $H$-orbit with 
    \begin{align*}
        \vol(Hg_1\Gamma) \ll p^{11\kappa_5N}.
    \end{align*}

    Let $D_0 = \max\{11\kappa_5, 21 + 11\kappa_7\}$, we get case (2). 

\end{proof}

\section{Margulis functions and random walks}\label{sec:MargulisFunctions}

The following is the main proposition of this section. 
\begin{proposition}\label{pro:MargulisMain}
    Let $0 < \eta < \eta_0$, $D \geq D_0 + 1$, and $x_0 \in X$, where $D_0$ is as in \cref{pro:MainClosingLemma}. Then there exists $N_0$ depending on $\eta$ and $X$ so that if $N \geq N_0$, at least one of the following holds: 
    \begin{enumerate}
        \item Let $0 < \epsilon < 10^{-70}$ and $0 < \alpha < 1$. Then there exists $x_1 \in X$, some $M$ with $9N \leq M \leq 9N + 2m_{\alpha}$ and a subset $F \subset B_{\mathfrak{r}}(0, 1)$ containing $0$ with 
        \begin{align*}
            p^N \leq \#F \leq p^{10N},
        \end{align*}
        so that both of the following holds: 
        \begin{enumerate}
            \item $\{\exp(w). x_1: w \in F_1\} \subset \bigl(K_{H, N/R} \cdot a_{M} \cdot \{u_r: r \in \Z_p\}. x_0\bigr)$, where $R > 0$ depends on $D$, $\epsilon$, and $\alpha$. 
            \item $\sum_{w' \neq w} \|w' - w\|_p \leq C_9 \cdot (\#F)^{1 + \epsilon}$ for all $w \in F$, where $C_9$ is an absolute constant. 
        \end{enumerate}
        \item There exists $x \in X$ such that $H.x$ is periodic with $\vol(H.x) \leq p^{D_0 N}$ and
        \begin{align*}
            x \in K[(D - D_0)N].x_0.
        \end{align*}
    \end{enumerate}
\end{proposition}
The proof of this proposition follows basically the same lines as in \cite[Section 7]{LM23}. We record the main arguments to ensure the proof works for this $p$-adic case. We claim no novelty in this section. Since we are always working in $\mathfrak{sl}_2(\Q_p)$, we make the convention that all norm $\|\cdot\|$ is $\|\cdot\|_p$ in this section. 

\subsection{The definition of a Margulis function}
We recall the definition of a Margulis function used in \cite[Section 7]{LM23} in this subsection. 

Let $F$ be a finite set and for every $w \in F$, there exist $x_w \in X$ and a bounded Borel set $\mathsf{E}_w \subset H$ satisfying the following:
\begin{enumerate}
    \item The map $\mathsf{h} \mapsto \mathsf{h}.x_w$ is injective on $\mathsf{E}_w$ for all $w \in F$. 
    \item $\mathsf{E}_w.x_w \cap \mathsf{E}_{w'}.x_{w'} = \emptyset$ for all $w \neq w'$.  
\end{enumerate}
Let $\mathcal{E} = \cup_{w \in F} \mathsf{E}_w. x_w$. Let $\mu_{\mathsf{E}_w}$ be the pushforward of the Haar measure on $H$ under the map $\mathsf{h} \mapsto \mathsf{h}.x_w$ and put
\begin{align*}
    \mu_{\mathcal{E}} = \frac{1}{\sum_{w \in F} m_{H}(\mathsf{E}_w)} \sum_{w \in F} \mu_{\mathsf{E}_w}.
\end{align*}

For every $(h, z) \in H \times \mathcal{E}$, define
\begin{align*}
    I_{\mathcal{E}}(h, z) := \bigl\{w \in \mathfrak{r}: 0 < \|w\|_p < \eta_0, \exp(w) h.z \in h.\mathcal{E}\bigr\}.
\end{align*}
Since $\mathsf{E}_w$ is bounded for all $w \in F$ and $F$ is finite, $I_{\mathcal{E}}(h, z)$ is finite for all $(h, z) \in H \times \mathcal{E}$. 

Fix $0 < \alpha < 1$. Define the Margulis function $f_{\mathcal{E}} := f_{\mathcal{E}, \alpha} : H \times \mathcal{E} \to [0, \infty)$ as following: 
\begin{align*}
    f_{\mathcal{E}}(h, z) = \begin{cases}
        \sum_{w \in I_{\mathcal{E}}(h, z)} \|w\|_p^{-\alpha} & \text{ if } I_{\mathcal{E}}(h, z) \neq \emptyset\\
        \eta_0^{-\alpha} & \text{otherwise}
    \end{cases}.
\end{align*}

Let $\nu = \nu_{\alpha}$ be the probability measure on $H$ defined by
\begin{align*}
    \nu(\varphi) = \int_{\Z_p} \varphi(a_{m_{\alpha}} u_r) \, dr.
\end{align*}
We will use $\nu^{(j)}$ to denote the $j$-fold convolution of $\nu$ for all $j \in \N$.  

Define $\psi_{\mathcal{E}}$ on $H \times \mathcal{E}$ by 
\begin{align*}
    \psi_{\mathcal{E}}(h, z) := \max\{\#I_{\mathcal{E}}(h, z), 1\} \eta_0^{-\alpha}.
\end{align*}

We recall the following lemma from \cite[Lemma 7.1]{LM23}. 
\begin{lemma}\label{lem:MargulisGeneral}
    There exists $C_{10} = C_{10}(\alpha)$ so that for all $\ell \in \N$ and all $z \in \mathcal{E}$, we have
    \begin{align*}
        \int f_{\mathcal{E}}(h, z)\, d\nu^{(\ell)}(h) \leq p^{-\ell} f_{\mathcal{E}}(e, z) + C_{10} \sum_{j = 1}^{\ell} p^{j - \ell} \int \psi_{\mathcal{E}}(h, z) \, d\nu^{(j)} (h).
    \end{align*}
\end{lemma}

\begin{proof}
    Use \cref{lem:LinearAlgebra} iterately. For a comprehensive proof, see \cite[Lemma 7.1]{LM23}. 
\end{proof}

\subsection{Some preparatory lemmas}We collect some preparatory lemmas in this subsection. 

Let $0 < \eta \leq \eta_0$ and $0 < \beta \leq \eta^2$. Define
\begin{align*}
    \mathsf{E} = K_{H, \beta} \cdot \bigl\{u_r: |r|_p \leq \frac{1}{p} \eta \bigr\}.
\end{align*}

Let $F \subset B_{\mathfrak{r}}(0, \beta)$ be a finite set, and let $y_0 \in X$. Then for all $w \in F$, we have $\mathsf{h} \mapsto \mathsf{h}. \exp(w) y_0$ is injective on $\mathsf{E}$. Put
\begin{align*}
    \mathcal{E} = \mathsf{E}.\{\exp(w).y_0: w \in F\}.
\end{align*}

The following lemma provides estimate on $\#I(a_{m}u_r, z)$ for $r \in \Z_p$. 

\begin{lemma}\label{lem:NumberOfSheet}
   There exists $C_{11} > 0$ so that for all $m \in \N$, all $r \in \Z_p$, and all $z \in \mathcal{E}$, we have
   \begin{align*}
       \#I_{\mathcal{E}}(a_m u_r, z) \leq C_{11} \beta^{-1} p^m \#F
   \end{align*}
\end{lemma}

\begin{proof}    
    Note that for all $z \in \mathcal{E}$ and $w \in I_{\mathcal{E}}(a_m u_r, z)$, we have
    \begin{align*}
        \exp(w) a_m u_r z \in a_m u_r \mathcal{E}.
    \end{align*}

    By \cref{lem:PropertyOfQThickening}, we have
    \begin{align*}
        Q_{\beta, m}^H a_m u_r K_{H, \beta} \subset a_m u_r K_{H, \beta}
    \end{align*}
    which implies that the map $(\mathsf{h}, w') \mapsto \mathsf{h} \exp(w') a_m u_r.z$ is injective over $ Q_{\beta, m}^{H} \times B_{\mathfrak{r}}(0, \eta_X)$. 

    Now we have
    \begin{align*}
        \beta^2 \eta \#F \gg  \sum_{w \in F} a_m u_r. m_{\mathsf{E}_w}(Q_{\beta, m}^{H}\exp(w).z) \gg p^{-m} \beta^{3} \#I_{\mathcal{E}}(a_m u_r, z),
    \end{align*}
    which shows the claim. 
\end{proof}

The following lemma enable us to compare the energy function and the Margulis function. 

\begin{lemma}\label{lem:MargulisDimensionEnergy}
    Let the notation be as above. Let $w_0 \in F$, then 
    \begin{align*}
        \sum_{w \neq w_0, w \in F} \|w - w_0\|^{-\alpha} \leq f_{\mathcal{E}}(e, z),
    \end{align*}
    where $z = \exp(w_0).y_0$. 
\end{lemma}
\begin{proof}
    For all $w \in F \subset B_{\mathfrak{r}}(0, \beta)$, we have
    \begin{align*}
        \exp(w). y_0 ={}& \exp(w) \exp(-w_0) \exp(w_0). y_0\\
        ={}& \exp(w') z. 
    \end{align*}

    By \cref{lem:BCH}, we have $\|w'\|_p = \|w - w_0\|_p$, which proves the claim. 
\end{proof}

\subsection{Dimension increasing}We will follow \cite[Section 7.2]{LM23} in this subsection. 

We first show that the 'discretized' dimension 
in transverse direction increase in an average way. 

\begin{lemma}\label{lem:MargulisDimensionIncreaseAverage}
    There exists $0 < \kappa_8 = \kappa_8(\alpha) \leq \frac{1}{m_{\alpha}}$ and $N_0$ depending on $X$ so that the following holds. Let $\mathcal{E}$ be defined as in the above subsection. Assume that
    \begin{align*}
        f_{\mathcal{E}}(e, z) \leq p^{BN} \text{ for all } z \in \mathcal{E}
    \end{align*}
    for some positive integers $B$ and $N$. 
    Then for all $0 < \epsilon < 0.1$ and all $\beta \geq p^{-\epsilon N/100}$, at least one of the following holds. 
    \begin{enumerate}
        \item $p^{BN} < p^{\frac{\epsilon N}{2}} (\#F)$. 
        \item For all integers $0 < \ell \leq \kappa_8 \epsilon n$ and all $z \in \mathcal{E}$, we have
        \begin{align*}
            \int f_{\mathcal{E}}(h, z) \, d\nu^{(\ell)}(h) \leq 2 \cdot p^{BN - \ell}.
        \end{align*}
    \end{enumerate}
\end{lemma}

\begin{proof}
    By \cref{lem:MargulisGeneral}, we have
    \begin{align*}
        \int f_{\mathcal{E}}(h, z) \, d\nu^{(\ell)}(h) \leq p^{-\ell} f_{\mathcal{E}}(e, z) + C_{10}\sum_{j = 1}^{\ell} p^{j - \ell} \int \psi_{\mathcal{E}}(h, z) \, d\nu^{(j)}(h).
    \end{align*}

    By \cref{lem:NumberOfSheet}, we have
    \begin{align*}
        \psi_{\mathcal{E}}(h, z) \leq{}& C_{11} \beta^{-1} p^{jm_{\alpha}} \eta_{X}^{-1} (\#F)\\
        \leq{}& C_{11} \beta^{-2} p^{jm_{\alpha}} (\#F)
    \end{align*}
    for all $h \in \supp \nu^{(j)}$. 

    Therefore, there exists $C > 0$ depending only on $m_{\alpha}$ such that if $j \leq \frac{\epsilon N}{C}$, we have
    \begin{align*}
        \psi_{\mathcal{E}}(h, z) \leq (pC_{10})^{-1} p^{\frac{\epsilon N}{4}} (\#F).
    \end{align*}

    Let $\kappa_8 = 1/C$, and let $\ell \leq \kappa_8\epsilon N$. Then 
    \begin{align}
        \int f_{\mathcal{E}}(h, z) \, d\nu^{(\ell)}(h) \leq p^{-\ell} f_{\mathcal{E}}(e, z) + p^{\frac{\epsilon N}{4}}(\#F) \leq p^{BN - \ell} + p^{\frac{\epsilon N}{4}}(\#F).
    \end{align}
    Therefore, either property~(1) holds, or $\#F \leq p^{BN - \frac{\epsilon}{2}}$, which implies
    \begin{align*}
        \int f_{\mathcal{E}}(h, z) \, d\nu^{(\ell)}(h) \leq p^{BN - \ell} + p^{BN - \frac{\epsilon}{4}} \leq 2 \cdot p^{BN - \ell}.
    \end{align*}
    The last inequality follows from the fact that $p^{\ell} \leq p^{\kappa_8\epsilon N} \leq p^{\epsilon N /4}$.
\end{proof}

From here to \cref{lem:MargulisInitialDimension}, we fix some $0 < \epsilon < 0.01$, and let $\beta = p^{-\kappa n/2}$ for some $0 < \kappa \leq 0.01\kappa_8 \epsilon$ which will be explicated later. The following lemma will convert the estimate we get on average in \cref{lem:MargulisDimensionIncreaseAverage} into pointwise estimate on at most points. 

\begin{lemma}\label{lem:MargulisChebyshev}
    Let the notation be the same as \cref{lem:MargulisDimensionIncreaseAverage}. Let $0 < \epsilon < 0.1$. Assume that 
    \begin{align*}
        \ell = \lfloor \kappa_8 \epsilon n\rfloor \geq 9|\log_p \eta|.
    \end{align*}
    Further assume \cref{lem:MargulisDimensionIncreaseAverage} property~(2) holds. 

    There exists $L_{\mathcal{E}} \subset \supp \nu^{(\ell)}$ with $\nu^{(\ell)}(L_{\mathcal{E}}) \geq 1 - p^{-\frac{\ell}{8}} \geq 1 - \eta$ so that both the following holds. 
    \begin{enumerate}
        \item For all $h_0 \in L_{\mathcal{E}}$, we have
        \begin{align}
            \int f_{\mathcal{E}}(h_0, z)\, d\mu_{\mathcal{E}}(z) \leq p^{BN - \frac{7\ell}{8}}.
        \end{align}
        \item For all $h_0 \in L_{\mathcal{E}}$, there exists $\mathcal{E}(h_0) \subset \mathcal{E}$ with $\mu_{\mathcal{E}}(\mathcal{E}(h_0)) \geq 1 - p^{-\frac{\ell}{8}} \geq 1 - \eta$ such that for all $z \in \mathcal{E}(h_0)$, we have
        \begin{enumerate}
            \item $K_{H, \beta}. z \subset \mathcal{E}$.
            \item $f_{\mathcal{E}}(h_0, z) \leq p^{BN - \frac{3\ell}{4}}$. 
        \end{enumerate}
    \end{enumerate}
\end{lemma}

\begin{proof}
    Both properties follows directly from Chebyshev inequality. See \cite[Lemma 7.6]{LM23}. 
\end{proof}

In the remaining part of this section, we will write $\mathsf{Q}_H$ for
\begin{align}
    \mathsf{Q}^{H}_{\beta, \ell m_0} = \{u_s^-: |s|_p \leq p^{-\ell m_0}\beta\} \cdot \{d_{\lambda}: |\lambda - 1|_p \leq \beta\} \cdot \{u_r: |r|_p \leq \beta\}
\end{align}
where $\ell = \kappa_8 \epsilon N$. Put
\begin{align*}
    \mathsf{Q}^G := \mathsf{Q}^H \exp(B_{\mathfrak{r}}(0, \beta)).
\end{align*}

\begin{lemma}\label{lem:MargulisCOvering}
    There exists a covering $\{\mathsf{Q}^G.y_j\}_{j \in \mathcal{J}}$ for $X$ where $\#\mathcal{J} \ll \beta^{-6} p^{\ell m_{\alpha}}$ and the implied constant depends only on $X$. 

    Moreover, for $h_0 \in L_{\mathcal{E}}$ we let 
    \begin{align}
        \mathcal{J}(h_0) = \{j \in \mathcal{J}: h_0. \mu_{\mathcal{E}} (h_0 \mathcal{E} \cap \mathsf{Q}^G.y_j) \geq \beta^{7} p^{-\ell m_{\alpha}}\}
    \end{align}
    and define $\hat{\mathcal{E}}(h_0) \subset \mathcal{E}(h_0)$ by 
    \begin{align*}
        h_0 \hat{\mathcal{E}(h_0)} = h_0 \mathcal{E}(h_0) \cap (\cup_{j \in \mathcal{J}(h_0)} \mathsf{Q}^G.y_j),
    \end{align*}
    then $\mu_{\mathcal{E}}(\hat{\mathcal{E}}(h_0)) \geq 1 - \beta$. 
\end{lemma}

\begin{proof}
    Since $\mathsf{Q}^H$ is subgroup of $K_{H, \beta}$, we have that $K_{H, \beta}$ is a disjoint union of $p^{\ell m_{\alpha}}$ many translation of $\mathsf{Q}^H$. The rest of the proof follows from a standard pigeonhole argument. See \cite[Lemma 7.6]{LM23}. 
\end{proof}

The following lemma yields a $\mathcal{E}_1$ for some $y_1$ and $F_1$, and with an improved bound on $f_{\mathcal{E}_1}(e, z)$.  
\begin{lemma}\label{lem:MargulisDimensionIncrease}
There exists $N_0 > 0$ so that the following holds for all $N \geq N_0$. Let the notation be as in \cref{lem:MargulisChebyshev} and \cref{lem:MargulisCOvering}. In particular, $0 < \epsilon < 0.01$ and 
\begin{align*}
    \ell = \lfloor \kappa_8 \epsilon N \rfloor \geq 9|\log_p \eta|;
\end{align*}
assume further that $\#F \geq p^N$ and that \cref{lem:MargulisDimensionIncreaseAverage}~(2) holds. 

Let $h_0 \in L_{\mathcal{E}}$ and let $y = y_j$ for some $j \in \mathcal{J}(h_0)$. There exists some 
\begin{align*}
    h_0 z_1 \in h_0 \mathcal{E}(h_0) \cap \mathsf{Q}^G. y
\end{align*}
and a subset 
\begin{align*}
    F_1 \subset B_{\mathfrak{r}}(0, \beta) \text{ with } \#F = \lceil \beta^7 \cdot (\#F) \rceil
\end{align*}
containing $0$, so that both of the following are satisfied. 
\begin{enumerate}
    \item For all $w \in F_1$, we have
    \begin{align*}
        \exp(w) h_0 z_1 \in K_{H, \beta}. h_0 \mathcal{E}(h_0).
    \end{align*}
    \item If we define $\mathcal{E}_1 = \mathsf{E}. \{\exp(w)h_0.z_1: w \in F_1\}$, then at least one of the following holds
    \begin{enumerate}
        \item $f_{\mathcal{E}_1}(e, z) \leq (\#F_1)^{1 + \epsilon}$ for all $z \in \mathcal{E}_1$.
        \item $f_{\mathcal{E}_1}(e, z) \leq p^{(B - \frac{3\kappa_8\epsilon}{4})N}$ for all $z \in \mathcal{E}_1$.
    \end{enumerate}
\end{enumerate}
\end{lemma}

\begin{proof}
    Let $h_0$ and $y = y_j$ be as in the statement. Note that the set $h_0. \mathcal{E}(h_0) \cap \mathsf{Q}^G.y$ is a union of local $H$-orbit. Let $B' \in \mathbb{N}$ be the smallest integer such that 
    \begin{align}
        h_0. \mathcal{E}(h_0) \cap \mathsf{Q}^G.y \subset \bigsqcup_{i = 1}^{B'} \mathsf{Q}^H \exp(w_i).y,
    \end{align}
    where $w \in B_{\mathfrak{r}}(0, \beta)$. 

    For all $1 \leq i \leq B'$, let $z_i \in \mathcal{E}(h_0)$ such that $h_0. z_i \in \mathsf{Q}^G.y$, and
    \begin{align*}
        h_0.z_i = \mathsf{h}_i \exp(w_i).y
    \end{align*}
    for some $\mathsf{h}_i \in \mathsf{Q}^H$. Such $z_i$ always exists since we are picking smallest $B'$. 

    Using \cref{lem:LocalProduct}, we have the following two properties. 
    \begin{enumerate}
        \item $\mathsf{Q}^H h_0 . z_i \cap \mathsf{Q}^H h_0 . z_j = \emptyset$ for $1 \leq i \neq j \leq B'$. 
        \item $h_0\mathcal{E}(h_0) \cap \mathsf{Q}^G. y \subset \bigcup_{i = 1}^{B'} \mathsf{Q}^H \cdot (\mathsf{Q}^H)^{-1} h_0. z_i$. 
    \end{enumerate}

    Now we give a lower bound for $M$. By the definition of $\mathcal{J}(h_0)$, we have
    \begin{align*}
        h_0.\mu_{\mathcal{E}}(h_0. \mathcal{E}(h_0) \cap \mathsf{Q}^G.y) \geq \beta^{7} p^{-\ell m_{\alpha}}.
    \end{align*}
    Therefore, we have
    \begin{align*}
        \sum_{i = 1}^{B'} h_0.\mu_{\mathcal{E}}(\mathsf{Q}^H \exp(w_i).y) \geq h_0.\mu_{\mathcal{E}}(h_0. \mathcal{E}(h_0) \cap \mathsf{Q}^G.y) \geq \beta^{7} p^{-\ell m_{\alpha}},
    \end{align*}
    which implies
    \begin{align*}
        \sum_{i = 1}^{B'} \beta^3 p^{-\ell m_{\alpha}} \beta^{-2} \eta^{-1} (\#F)^{-1} \gg \beta^{7} p^{-\ell m_{\alpha}}
    \end{align*}

    Enlarging $n$, we have
    \begin{align}
        B' \geq \beta^7 \cdot (\#F).
    \end{align}

    For $1 \leq i, j \leq B'$, we have
    \begin{align}
        h_0.z_i ={}& \mathsf{h}_i \exp(w_i).y\\
        ={}& \mathsf{h}_i \exp(w_i) \exp(-w_j) \mathsf{h}_j^{-1} h_0.z_j\\
        ={}& \mathsf{h}_i \mathsf{h}_j^{-1}\exp(w_{ij})h_0.z_j,
    \end{align}
    where $\mathsf{h}_i, \mathsf{h}_j \in \mathsf{Q}^H$ and $\|w_{ij}\|_p = \|w_i - w_j\|_p$ by \cref{lem:BCH}. 

    Let $F_1 \subset \{w_{i1}: 1 \leq i \leq B'\}$ where $\#F_1 = \lceil \beta^{7}(\#F)\rceil$. Let $\mathcal{E}_1 = K_{H, \beta} \{\exp(w)h_0.z_1 : w \in F_1\}$. 

    We now show property~(1). For all $w \in F_1$, $w = w_{i1}$ for some $1 \leq i \leq M$. Hence we have
    \begin{align*}
        \exp(w) h_0.z_1 ={}&  \exp(w_{i1}) h_0.z_1\\
        ={}& \mathsf{h}_1 \mathsf{h}_i^{-1} h_0.z_i\\
        \in{}& K_{H, \beta} h_0 \mathcal{E}(h_0).
    \end{align*}

    Now we show property~(2).We want to compare $f_{\mathcal{E}_1}(e, z)$ for $z \in \mathcal{E}_1$ with $f_{\mathcal{E}}(h_0, z_i)$ with $z \in \mathsf{E} \exp(w_{i1}) h_0.z_1$. 

    For all $z \in \mathcal{E}_1$ and $w \in I_{\mathcal{E}_1}(e, z)$, we have
    \begin{align*}
        z ={}& h u_r \exp(w_{i1}) h_0. z_1\\
        ={}& h u_r \mathsf{h}_1 \mathsf{h}_i^{-1} h_0. z_i
    \end{align*}
    for some $h \in K_{H, \beta}$ and $|r|_p \leq \eta$ and
    \begin{align*}
        \exp(w).z ={}& h' u_{r'} \exp(w_{j1}) h_0.z_1\\
        ={}& h' u_{r'} \mathsf{h}_1 \mathsf{h}_j^{-1} h_0. z_j
    \end{align*}
    for some $h' \in K_{H, \beta}$ and $|r'|_p \leq \eta$. 

    Now we have
    \begin{align*}
        \exp(w) h u_r \mathsf{h}_1 \mathsf{h}_i^{-1} h_0. z_i = h' u_{r'} \mathsf{h}_1 \mathsf{h}_j^{-1} h_0. z_j
    \end{align*}
    which implies
    \begin{align*}
        \exp(w) h u_r \mathsf{h}_1 \mathsf{h}_i^{-1} h_0. z_i ={}& h' u_{r'} \mathsf{h}_1 \mathsf{h}_j^{-1} \mathsf{h}_j \mathsf{h}_i^{-1} \exp(w_{ji}) h_0. z_i\\
        ={}& h' u_{r'} \mathsf{h}_1  \mathsf{h}_i^{-1} \exp(w_{ji}) h_0. z_i. 
    \end{align*}

    By \cref{lem:BCH}, we have $\|w\|_p = \|w_{ji}\|_p$. 

    Now we show that $w_{ji} \in I_{\mathcal{E}}(h_0, z_i)$. By the definition of $w_{ji}$, we have
    \begin{align*}
        \exp(w_{ji}) h_0.z_i ={}& \mathsf{h}_i \mathsf{h}_j^{-1} h_0.z_j\\
        ={}& h_0 h_0^{-1}\mathsf{h}_i h_0 h_0^{-1} \mathsf{h}_j^{-1} h_0.z_j.
    \end{align*}
    Since $\mathsf{h}_i, \mathsf{h}_j \in \mathsf{Q}^H$, we have $h_0^{-1}\mathsf{h}_i h_0, h_0^{-1} \mathsf{h}_j^{-1} h_0 \in K_{H, \beta}$, which shows $\mathsf{h}_i \mathsf{h}_j^{-1} h_0.z_j \in h_0 \mathcal{E}$. 

    Therefore, we have
    \begin{align*}
        f_{\mathcal{E}_1}(e, z) \leq f_{\mathcal{E}}(h_0, z) \leq p^{(B - \frac{3 \kappa_8\epsilon}{4})N}.
    \end{align*}
\end{proof}

We also have the following lemma providing the base case for our inductive argument. 

\begin{lemma}\label{lem:MargulisInitialDimension}
    Let the notation be as in \cref{pro:MargulisMain}. In particular, let $0 < \eta < \eta_0$, $D \geq D_0$, and $x_0 \in X$. There exists $N_1$, depending on $\eta$, $D$, and $X$, so that the following holds for $N \geq N_1$. 

    Let $0 < \epsilon < 10^{-70}$, and let $\beta = p^{-\kappa(N + 1)/2}$ where $0 < \kappa \leq \frac{1}{100}\kappa_8\epsilon$. Then at least one of the following holds. 
    \begin{enumerate}
        \item There exists $F \subset B_{\mathfrak{r}}(0, \beta)$ with 
        \begin{align*}
            p^{2N - 5\kappa(N + 1)} \leq \#F \leq p^{2N + \kappa(N + 1)/2}
        \end{align*}
        and some $y \in \bigl(K_{H, \beta} \cdot a_{4N}\bigr) \cdot \{u_r: r\in \Z_p\}.x_0$ so that if we put 
        \begin{align*}
            \mathcal{E} = \mathsf{E}.\{\exp(w).y: w \in F\},
        \end{align*}
        then $\mathcal{E} \subset \bigl(K_{H, \beta} \cdot a_{5N}\bigr) \cdot \{u_r: r\in \Z_p\}.x_0$ and 
        \begin{align*}
            f_{\mathcal{E}}(e, z) \leq p^{DN}
        \end{align*}
        for all $z \in \mathcal{E}$. 
        \item There exists $x \in X$ such that $H.x$ is periodic with $\vol(H.x) \leq p^{D_0 N}$ and
        \begin{align*}
            x \in K[(D - D_0)N].x_0.
        \end{align*}
    \end{enumerate}
\end{lemma}

\begin{proof}
    Let $\mathcal{C}_0 = \{a_{4N}u_r.x_0: r \in \Z_p\}$. Apply \cref{pro:MainClosingLemma} with $x_0$ and $N$, if property~(2) in \cref{pro:MainClosingLemma} holds, then property~(2) in \cref{lem:MargulisInitialDimension} holds. Now we assume property~(1) in \cref{pro:MainClosingLemma} holds. 
    
    Let $x$ be a point given by \cref{pro:MainClosingLemma}~(1). Let $\mathcal{C} = K_{H, \beta}a_{N} \{u_r: r \in \Z_p\}.x$. Let $\mathsf{C} =  K_{H, \beta}a_{N} \{u_r: r \in \Z_p\}$. Let $\mu_{\mathcal{C}}$ be the pushforward of the normalized measure on $\mathsf{C}$. Note that we are using different notations here to avoid confusion with $\mathsf{E} = K_{H, \beta} \cdot \{u_r: |r|_p \leq \eta\}$ in this section. 

    Let $\{K_{\beta}.\hat{y}_j\}_{j \in \mathcal{J}}$ be a disjoint cover of $X$. We have $\#\mathcal{J} \asymp \beta^{-6}$ where the implied constant depends only on $X$. Let $\mathcal{J}'$ be the set of $j \in \mathcal{J}$ such that 
    \begin{align*}
        \mu_{\mathcal{C}}(\mathcal{C} \cap K_{\beta}.\hat{y}_j) \geq \beta^{7}.
    \end{align*}
    We have
    \begin{align*}
        \mu_{\mathcal{C}}\bigl( \mathcal{C} \cap \bigl(\bigcup_{j \in \mathcal{J}'} K_{\beta}.\hat{y}_j\bigr)\bigr) \geq 1 - \beta
    \end{align*}

    Pick $j \in \mathcal{J}'$, let $\hat{y} = \hat{y}_j$. Then we have $w_i \in B_{\mathfrak{r}}(0, \beta)$ and $\mathsf{h}_i \in K_{H, \beta}$, $1 \leq i \leq B'$ so that $\mathsf{h}_i \exp(w_i) \hat{y} \in \mathcal{C}$ and 
    \begin{align*}
        \mathcal{C} \cap K_{\beta}.\hat{y} = \bigcup_{i = 1}^{B'}\mathsf{C}_i\mathsf{h}_i \exp(w_i).\hat{y},
    \end{align*}
    where $\mathsf{C}_i \subset K_{H, \beta}$. 

    Now we estimate $B'$. Note that
    \begin{align*}
        \mu_{\mathcal{C}}(K_{H, \beta}) \ll \beta^3 (p^{2N}\beta^2)^{-1} = \beta p^{-2N},
    \end{align*}
    which implies $B' \gg \beta^6 p^{2N}$. Enlarging $N$, we get 
    \begin{align*}
        B' \geq \beta^7 p^{2N}.
    \end{align*}

    Now we construct $F$ and $\mathcal{E}$. Note that for every $1 \leq i, j \leq B'$, we have
    \begin{align*}
        \mathsf{h}_i \exp(w_i).\hat{y} ={}&  \mathsf{h}_i \exp(w_i)\exp(-w_j)\mathsf{h}_j^{-1} \mathsf{h}_j \exp(w_j).\hat{y}\\
        ={}& \mathsf{h}_i\mathsf{h}_j^{-1} \exp(w_{ij}) \mathsf{h}_j \exp(w_j) \hat{y},
    \end{align*}
    where $\|w_{ij}\|_p = \|w_i - w_j\|_p$ by \cref{lem:BCH} and \cref{lem:LocalProduct}. We also have
    \begin{align*}
        \exp(w_{ij})\mathsf{h}_j\exp(w_j).\hat{y} ={}& \mathsf{h_j} \mathsf{h}_i^{-1} \mathsf{h}_i\exp(w_i).\hat{y}\\
        \in{}& \mathsf{h_j} \mathsf{h}_i^{-1} \mathcal{C} \subset \mathcal{C}.
    \end{align*}

    Let $y = \mathsf{h}_1\exp(w_1).\hat{y}$ and $F = \{w_{i1}: 1 \leq i \leq B'\}$. By \cref{lem:ClosingLemmaNumberOfSheet}, we have 
    \begin{align*}
        \#F \ll p^{2N} \leq \beta^{-1} p^{2N}
    \end{align*}
    by letting $\beta$ small enough. Therefore
    \begin{align}
        p^{2N - 5\kappa(N + 1)} = \beta^7 p^{2N} \leq \#F = B' \leq \beta^{-1} p^{2N} = p^{2N + \kappa(N + 1)/2}.
    \end{align}

    Define $\mathcal{E} = \mathsf{E}.\{\exp(w_{i1}).y: w_{i1} \in F\}$. Using the fact that $K_{H, \beta}$ is a normal subgroup of $K_H$ and a straight forward calculation, we have 
    \begin{align*}
        \mathcal{E} \subset \bigl(K_{H, \beta} \cdot a_{N}\bigr) \cdot \bigl\{ u_r: r \in \Z_p\bigr\}.x. 
    \end{align*}

    Since $x \in \{a_{4N} u_r.x_0: r \in \Z_p\}$, we have
    \begin{align*}
        \mathcal{E} \subset \bigl(K_{H, \beta} \cdot a_{5N}\bigr) \cdot \{u_r: r\in \Z_p\}.x_0.
    \end{align*}

    Now we estimate $f_{\mathcal{E}}(e, z)$ for all $z \in \mathcal{E}$. Let $z_i = \mathsf{h}_i\exp(w_i).\hat{y}$ and $z = h u_r\exp(w_{i1}).y$ for $h \in K_{H, \beta}$ and $|r|_p \leq \eta$. Pick $w \in I_{\mathcal{E}}(e, z)$, we have
    \begin{align*}
        \exp(w).z = h' u_{r'}\exp(w_{j1}).y
    \end{align*}
    for $h' \in K_{H, \beta}$ and $|r'|_p \leq \eta$. We want to compare $f_{\mathcal{E}}(e, z)$ with $f_{\mathcal{C}} (e, z_i)$. 

    Note that 
    \begin{align*}
        z ={}& h u_r\exp(w_{i1}).y\\
        ={}& h u_r \exp(w_{i1}) \mathsf{h}_1\exp(w_1).\hat{y}\\
        ={}& h u_r \mathsf{h}_1 \mathsf{h}_i^{-1} z_{i}.
    \end{align*}
    \begin{align*}
        \exp(w).z ={}& h' u_{r'}\exp(w_{j1}).y\\
        ={}& h' u_{r'}\exp(w_{j1}) \mathsf{h}_1\exp(w_1).\hat{y}\\
        ={}& h' u_{r'}\mathsf{h}_1 \mathsf{h}_j^{-1} z_{j}.
    \end{align*}
    Therefore, we have
    \begin{align*}
        \exp(w) h u_r \mathsf{h}_1 \mathsf{h}_i^{-1} z_{i} ={}&  h' u_{r'}\mathsf{h}_1 \mathsf{h}_j^{-1} z_{j}\\
        ={}& h' u_{r'}\mathsf{h}_1 \mathsf{h_i}^{-1} \exp(w_{ji}).z_{i}.
    \end{align*}
    By \cref{lem:BCH}, $\|w_{ji}\|_p = \|w\|_p$. 

    We claim that $w_{ji} \in I_{\mathcal{C}}(e, z_i)$. 
    Note that
    \begin{align*}
        \exp(w_{ji}).z_i = \mathsf{h}_i \mathsf{h}_{j}^{-1} z_j \in K_{H, \beta}\mathcal{C} = \mathcal{C}.
    \end{align*}

    Therefore, we have
    \begin{align*}
        f_{\mathcal{E}}(e, z) \leq f_{\mathcal{C}}(e, z_i) \leq p^{DN}.
    \end{align*}
\end{proof}

\begin{proof}[Proof of \cref{pro:MargulisMain}]
    We give a sketch of the proof here. For a detailed proof, see \cite[Proposition 7.1]{LM23}. 
    \begin{enumerate}
        \item We first use \cref{lem:MargulisInitialDimension}. If \cref{lem:MargulisInitialDimension}~(2) holds, then \cref{pro:MargulisMain}~(2) holds, which completes the proof. If not, by \cref{lem:MargulisInitialDimension}~(1), we could construct sets $\mathcal{E}_0$ and $F_0$. Now we use \cref{lem:MargulisDimensionIncreaseAverage} to this $\mathcal{E}_0$. If \cref{lem:MargulisDimensionIncreaseAverage}~(1) holds, then we have dimension close to $1$ at the beginning, which completes the proof. 
        \item Now suppose \cref{lem:MargulisDimensionIncreaseAverage}~(2) holds for $\mathcal{E}_0$. Let $L_{\mathcal{E}_0}$ be as in \cref{lem:MargulisChebyshev}. Let $h_0 \in L_{\mathcal{E}_0}$ and let $y_j$ for some $j \in \mathcal{J}(h_0) $ as in \cref{lem:MargulisCOvering}. By \cref{lem:MargulisDimensionIncrease}, there exists $z_1$ such that $h_0.z_1 \in h_0\mathcal{E}(h_0) \cap \mathsf{Q}^G$, $F_1 \subset B_{\mathfrak{r}}(0, \beta)$ containing $0$ with the following properties:
        \begin{enumerate}
            \item $\#F_1 \geq \lceil \beta^7 \cdot (\#F_0)\rceil$. 
            \item For all $w \in F_1$, we have
            \begin{align*}
                \exp(w) h_0.z_1 \in K_{H, \beta} h_0 \mathcal{E}(h_0).
            \end{align*}
            \item Let $\mathcal{E}_1 = \mathsf{E}.\{\exp(w)h_0z_1: w \in F_1\}$, then at least one of the following holds: 
            \begin{enumerate}
                \item $f_{\mathcal{E}_1}(e, z) \leq (\#F)^{1 + \epsilon}$ for all $z \in \mathcal{E}_1$. 
                \item $f_{\mathcal{E}_1}(e, z) \leq p^{(B - \frac{3 \kappa_8\epsilon}{4})N}$
            \end{enumerate}
        \end{enumerate}

        If (c)~(i) holds, then the proof is completed. 

        Otherwise, we could repeat the construction to define $F_2, ...$ and corresponding $\mathcal{E}_2, ...$. 
        \item Let $i_{\max} = \lfloor \frac{4B - 3}{4\kappa_8\epsilon}\rfloor + 1$, then after are at most $i_{\max}$ many steps, we obtain a set $\mathcal{E}$ which satisfies \cref{pro:MargulisMain}~(1). 
    \end{enumerate}

\end{proof}

The proof of \cref{pro:Main} and \cref{thm:Main} will follows exactly as \cite[Section 8]{LM23} combining \cref{pro:MainClosingLemma}, \cref{pro:RestrictedProjMain}, \cref{thm:VenkateshTrick}, and \cref{pro:MargulisMain}. 

\begin{appendices}
\section{\texorpdfstring{Proof of \cref{pro:PAdicSobolev}}{Proof of Proposition 2.12}}\label{sec:Sobolev}
The proof is essentially contained in \cite[Appendix A]{EMMV20}, we include here for completeness. 
\begin{proof}[Proof of \cref{pro:PAdicSobolev}]We prove as in \cite[Appendix A]{EMMV20}. 

\noindent
\textit{Proof of property~(S1).}

Note that if $f$ is $K[m]$-invariant, we have
    \begin{align*}
        |f(x)|^2 = \frac{1}{\mathrm{Vol}(K[m])} \int_{K[m]} |f(k.x)|^2 dk \ll p^{m\dim X} \|f\|_2^2.
    \end{align*}

    Then for general locally constant compactly support $f$, we have 
    \begin{align*}
        |f(x)|^2 = |\sum_m \mathrm{pr}[m]. f (x)|^2 
        \leq{}& (\sum_m p^{-2m})(\sum_m p^{2m} |\mathrm{pr}[m]. f (x)|^2)\\ 
        \ll{}& \sum_m p^{(2 + \dim X)m} \|\mathrm{pr}[m]. f (x)\|_2^2\\
        ={}& \mathcal{S}_{\dim X + 2}^2(f). 
    \end{align*}
\noindent
\textit{Proof of property~(S2).}

Note that if $g \in K$, then $gK[m]g^{-1} = K[m]$. Hence $g\cdot \mathrm{pr}[m]. f = \mathrm{pr}[m](g \cdot f)$. Therefore $\mathcal{S}_d(g\cdot f) = \mathcal{S}_d(f)$ for all $g \in K$. 

If $g \notin K$, by direct calculation, we have
\begin{align*}
    g K[m + 2 \log_p\|g\|] g^{-1} \subset K[m]. 
\end{align*}

By $\mathrm{Av}[l - 1]\mathrm{pr}[l] = 0$, we have 
\begin{align*}
    \mathrm{pr}[m](g \cdot \mathrm{pr}[l]. f) = 0 \text{ unless } |m - l| \leq 2\log_p\|g\|. 
\end{align*}

Therefore, 
\begin{align*}
    \mathcal{S}_d(g\cdot f)^2 ={}& \sum_m p^{md} \|\mathrm{pr}[m](g \cdot \sum_{l} \mathrm{pr}[l] f)\|_2^2\\
    ={}& \sum_m p^{md} (4\log_p\|g\| + 1) \|\mathrm{pr}[m](g \cdot \max_{|l - m| \leq 2\log_p\|g\|}\mathrm{pr}[l] f)\|_2^2\\
    \ll{}& (4\log_p\|g\| + 1)^2 \|g\|^{2d}\mathcal{S}_d(f)^2.
\end{align*}
\noindent
\textit{Proof of property~(S3).}

Note that $m \leq r$, $g \cdot \mathrm{pr}[m]f = \mathrm{pr}[m]f$. 

We argue as in the proof of property~(S1). We have
\begin{align*}
    |(g \cdot f - f) (x)|^2 ={}& |\sum_m \mathrm{pr}[m](g \cdot f - f) (x)|^2\\
    ={}& |\sum_{m > r} \mathrm{pr}[m](g \cdot f - f)(x)|^2\\ 
    \ll{}& p^{-2r} \sum_m p^{2m} |\mathrm{pr}[m](g \cdot f - f)(x)|^2\\ 
    \ll{}& p^{-2r} \mathcal{S}_{\dim X + 2}(f)^2.
\end{align*}
\noindent
\textit{Proof of property~(S4).}

Note that if $l \leq m$, $\mathrm{pr}[m]((\mathrm{pr}[l].f_1) \cdot f_2) = \mathrm{pr}[l].f_1 \cdot \mathrm{pr}[m]. f_2$, we have
\begin{align*}
    \mathcal{S}_d(f_1f_2)^2 ={}& \sum_m p^{md} \|\mathrm{pr}[m]. (f_1f_2)\|_2^2\\
    ={}& \sum_m p^{md} \|\mathrm{pr}[m]. ((\sum_l \mathrm{pr}[l].f_1)\cdot f_2)\|_2^2\\
    \ll{}&  \sum_m p^{md} \|(\sum_{l \leq m} \mathrm{pr}[l].f_1)\cdot \mathrm{pr}[m]f_2\|_2^2 + \sum_m p^{md} \sum_{l > m} \|\mathrm{pr}[m]((\mathrm{pr}[l].f_1)\cdot f_2)\|_2^2\\
    \ll{}& \sum_m p^{md} \|\sum_{l \leq m} \mathrm{pr}[l].f_1\|_{\infty}^2\|\mathrm{pr}[m]. f_2\|_2^2 + \sum_m p^{md} \sum_{l > m} \|(\mathrm{pr}[l].f_1)\cdot f_2)\|_2^2\\
    \ll{}& \|f_1\|_\infty^2\mathcal{S}_d(f_2)^2 +  \sum_{l} (\sum_{m < l} p^{md}) \|\mathrm{pr}[l]f_1\|_2^2 \|f_2\|_{\infty}^2\\
    \ll{}& \|f_1\|_\infty^2\mathcal{S}_d(f_2)^2 + \|f_2\|_\infty^2\mathcal{S}_d(f_1)^2.
\end{align*}

By property~(S1), if $d \geq d_0$, we have $\mathcal{S}_d(f_1f_2) \ll \mathcal{S}_d(f_1)\mathcal{S}_d(f_2)$. 
\end{proof}

\end{appendices}

\nocite{*}
\bibliographystyle{alpha_name-year-title}
\bibliography{References}

\newcommand{\etalchar}[1]{$^{#1}$}
\begin{thebibliography}{LMWY23}

\bibitem[AG09]{AG09}
Avraham Aizenbud and Dmitry Gourevitch.
\newblock Generalized {H}arish-{C}handra descent, {G}elfand pairs, and an {A}rchimedean analog of {J}acquet-{R}allis's theorem.
\newblock {\em Duke Math. J.}, 149(3):509--567, 2009.
\newblock With an appendix by the authors and Eitan Sayag.

\bibitem[AHR20]{AHR20}
Jarod Alper, Jack Hall, and David Rydh.
\newblock A {L}una \'{e}tale slice theorem for algebraic stacks.
\newblock {\em Ann. of Math. (2)}, 191(3):675--738, 2020.

\bibitem[BGM19]{BGM19}
Michael Bate, Haralampos Geranios, and Benjamin Martin.
\newblock Orbit closures and invariants.
\newblock {\em Math. Z.}, 293(3-4):1121--1159, 2019.

\bibitem[BO12]{BO12}
Yves Benoist and Hee Oh.
\newblock Effective equidistribution of {S-integral} points on symmetric varieties.
\newblock {\em Annales de l'Institut Fourier}, 62(5):1889--1942, 2012.

\bibitem[Bor19]{Bor19}
Armand Borel.
\newblock {\em Introduction to arithmetic groups}, volume~73 of {\em University Lecture Series}.
\newblock American Mathematical Society, Providence, RI, 2019.
\newblock Translated from the 1969 French original [MR0244260] by Lam Laurent Pham, Edited and with a preface by Dave Witte Morris.

\bibitem[Bou89]{Bou89}
Nicolas Bourbaki.
\newblock {\em Lie groups and {L}ie algebras. {C}hapters 1--3}.
\newblock Elements of Mathematics (Berlin). Springer-Verlag, Berlin, 1989.
\newblock Translated from the French, Reprint of the 1975 edition.

\bibitem[Bou10]{Bou10}
Jean Bourgain.
\newblock The discretized sum-product and projection theorems.
\newblock {\em J. Anal. Math.}, 112:193--236, 2010.

\bibitem[BFLM11]{BFLM11}
Jean Bourgain, Alex Furman, Elon Lindenstrauss, and Shahar Mozes.
\newblock Stationary measures and equidistribution for orbits of nonabelian semigroups on the torus.
\newblock {\em J. Amer. Math. Soc.}, 24(1):231--280, 2011.

\bibitem[BG09]{BG09}
Jean Bourgain and Alex Gamburd.
\newblock Expansion and random walks in {${\mathrm SL}_d(\mathbb Z/p^n\mathbb Z)$}. {II}.
\newblock {\em J. Eur. Math. Soc. (JEMS)}, 11(5):1057--1103, 2009.
\newblock With an appendix by Bourgain.

\bibitem[CEG{\etalchar{+}}90]{CEHGLSW90}
Kenneth~L. Clarkson, Herbert Edelsbrunner, Leonidas~J. Guibas, Micha Sharir, and Emo Welzl.
\newblock Combinatorial complexity bounds for arrangements of curves and spheres.
\newblock {\em Discrete Comput. Geom.}, 5(2):99--160, 1990.

\bibitem[Clo03]{Clo03}
Laurent Clozel.
\newblock D{\'e}monstration de la conjecture $\tau$.
\newblock {\em Inventiones mathematicae}, 151:297--328, 2003.

\bibitem[EE93]{EE93}
B.~Edixhoven and J.-H. Evertse, editors.
\newblock {\em Diophantine approximation and abelian varieties}, volume 1566 of {\em Lecture Notes in Mathematics}.
\newblock Springer-Verlag, Berlin, 1993.
\newblock Introductory lectures, Papers from the conference held in Soesterberg, April 12--16, 1992.

\bibitem[EMMV20]{EMMV20}
M.~Einsiedler, G.~Margulis, A.~Mohammadi, and A.~Venkatesh.
\newblock Effective equidistribution and property {$(\tau)$}.
\newblock {\em J. Amer. Math. Soc.}, 33(1):223--289, 2020.

\bibitem[EMV09]{EMV09}
M.~Einsiedler, G.~Margulis, and A.~Venkatesh.
\newblock Effective equidistribution for closed orbits of semisimple groups on homogeneous spaces.
\newblock {\em Invent. Math.}, 177(1):137--212, 2009.

\bibitem[ELMV11]{ELMV11}
Manfred Einsiedler, Elon Lindenstrauss, Philippe Michel, and Akshay Venkatesh.
\newblock Distribution of periodic torus orbits and {D}uke's theorem for cubic fields.
\newblock {\em Ann. of Math. (2)}, 173(2):815--885, 2011.

\bibitem[GGG{\etalchar{+}}24]{GGGHMW24}
Shengwen Gan, Shaoming Guo, Larry Guth, Terence L.~J. Harris, Dominique Maldague, and Hong Wang.
\newblock On restricted projections to planes in $\mathbb{R}^3$, 2024.

\bibitem[GGW24]{GGW24}
Shengwen Gan, Shaoming Guo, and Hong Wang.
\newblock A restricted projection problem for fractal sets in $\mathbb{R}^n$, 2024.

\bibitem[GMO06]{GMO06}
Alex Gorodnik, Francois Maucourant, and Hee Oh.
\newblock Manin's and peyre's conjectures on rational points and adelic mixing.
\newblock {\em Annales Scientifiques de l'Ecole Normale Superieure}, 41, 02 2006.

\bibitem[JL]{JL}
Ben Johnsrude and Zuo Lin.
\newblock Restricted projections and fourier decoupling in $\mathbb{Q}_p^n$.
\newblock {\em In preparation}.

\bibitem[KST17]{KST17}
Dmitry Kleinbock, Ronggang Shi, and George Tomanov.
\newblock $\mathcal{S}$-adic version of minkowski's geometry of numbers and mahler's compactness criterion.
\newblock {\em J. Number Theory}, 174:150--163, 2017.

\bibitem[KOV21]{KOV21}
Antti Käenmäki, Tuomas Orponen, and Laura Venieri.
\newblock A marstrand-type restricted projection theorem in $\mathbb{R}^{3}$, 2021.

\bibitem[LM23]{LM23}
E.~Lindenstrauss and A.~Mohammadi.
\newblock Polynomial effective density in quotients of {$\mathbb H^3$} and {$\mathbb H^2\times\mathbb H^2$}.
\newblock {\em Invent. Math.}, 231(3):1141--1237, 2023.

\bibitem[LMW22]{LMW22}
Elon Lindenstrauss, Amir Mohammadi, and Zhiren Wang.
\newblock Effective equidistribution for some one parameter unipotent flows, 2022.

\bibitem[LMW23]{LMW23}
Elon Lindenstrauss, Amir Mohammadi, and Zhiren Wang.
\newblock Quantitative equidistribution and the local statistics of the spectrum of a flat torus, 2023.

\bibitem[LMWY23]{LMWY23}
Elon Lindenstrauss, Amir Mohammadi, Zhiren Wang, and Lei Yang.
\newblock An effective version of the oppenheim conjecture with a polynomial error rate, 2023.

\bibitem[Lun75]{Lun75}
D.~Luna.
\newblock Sur certaines op\'{e}rations diff\'{e}rentiables des groupes de {L}ie.
\newblock {\em Amer. J. Math.}, 97:172--181, 1975.

\bibitem[Mar91]{Mar91}
G.~A. Margulis.
\newblock {\em Discrete subgroups of semisimple {L}ie groups}, volume~17 of {\em Ergebnisse der Mathematik und ihrer Grenzgebiete (3) [Results in Mathematics and Related Areas (3)]}.
\newblock Springer-Verlag, Berlin, 1991.

\bibitem[Mil17]{Mil17}
J.~S. Milne.
\newblock {\em Algebraic groups}, volume 170 of {\em Cambridge Studies in Advanced Mathematics}.
\newblock Cambridge University Press, Cambridge, 2017.
\newblock The theory of group schemes of finite type over a field.

\bibitem[Neu99]{Neu99}
J\"{u}rgen Neukirch.
\newblock {\em Algebraic number theory}, volume 322 of {\em Grundlehren der mathematischen Wissenschaften [Fundamental Principles of Mathematical Sciences]}.
\newblock Springer-Verlag, Berlin, 1999.
\newblock Translated from the 1992 German original and with a note by Norbert Schappacher, With a foreword by G. Harder.

\bibitem[PR94]{PR94}
Vladimir Platonov and Andrei Rapinchuk.
\newblock {\em Algebraic groups and number theory}, volume 139 of {\em Pure and Applied Mathematics}.
\newblock Academic Press, Inc., Boston, MA, 1994.
\newblock Translated from the 1991 Russian original by Rachel Rowen.

\bibitem[SG17]{SG17}
Alireza Salehi~Golsefidy.
\newblock Super-approximation, {I}: {$\mathfrak{p}$}-adic semisimple case.
\newblock {\em Int. Math. Res. Not. IMRN}, (23):7190--7263, 2017.

\bibitem[Sch03]{Sch03}
W.~Schlag.
\newblock On continuum incidence problems related to harmonic analysis.
\newblock {\em J. Funct. Anal.}, 201(2):480--521, 2003.

\bibitem[Ser03]{Ser03}
Jean-Pierre Serre.
\newblock {\em Trees}.
\newblock Springer Monographs in Mathematics. Springer-Verlag, Berlin, 2003.
\newblock Translated from the French original by John Stillwell, Corrected 2nd printing of the 1980 English translation.

\bibitem[Ser06]{Ser06}
Jean-Pierre Serre.
\newblock {\em Lie algebras and {L}ie groups}, volume 1500 of {\em Lecture Notes in Mathematics}.
\newblock Springer-Verlag, Berlin, 2006.
\newblock 1964 lectures given at Harvard University, Corrected fifth printing of the second (1992) edition.

\bibitem[Ven10]{Ven10}
Akshay Venkatesh.
\newblock Sparse equidistribution problems, period bounds and subconvexity.
\newblock {\em Ann. of Math. (2)}, 172(2):989--1094, 2010.

\bibitem[Wol00]{Wol00}
T.~Wolff.
\newblock Local smoothing type estimates on {$L^p$} for large {$p$}.
\newblock {\em Geom. Funct. Anal.}, 10(5):1237--1288, 2000.

\bibitem[Yan23]{Yan23}
Lei Yang.
\newblock Effective version of ratner's equidistribution theorem for $\mathrm{SL}(3,\mathbb{R})$, 2023.

\end{thebibliography}
\end{document}